\documentclass[12pt]{amsart}%
\usepackage{amssymb,amsfonts}
\usepackage{amscd}
\usepackage{amsmath, mathtools}
\usepackage{amsfonts}
\usepackage{amssymb}
\usepackage{verbatim}
\usepackage{xcolor}
\usepackage{bm, enumerate}

\newtheorem{theorem}{Theorem}[section]
\newtheorem{lemma}[theorem]{Lemma}
\newtheorem{proposition}[theorem]{Proposition}

\newtheorem{example}[theorem]{Example}
\newtheorem{definition}[theorem]{Definition}
\newtheorem{corollary}[theorem]{Corollary}
\newtheorem{remark}[theorem]{Remark}

\newcommand{\be}{\begin{equation}}
\newcommand{\ee}{\end{equation}}

\newcommand{\cB}{\mathcal{B}}

\newcommand{\cD}{\mathcal{D}}
\newcommand{\cE}{\mathcal{E}}
\newcommand{\cM}{\mathcal{M}}
\newcommand{\cF}{\mathcal{F}}
\newcommand{\cL}{\mathcal{L}}

\newcommand{\cP}{\mathcal{P}}

\newcommand{\cT}{\mathcal{T}}

\newcommand{\bC}{\mathbb{C}}

\newcommand{\bN}{\mathbb{N}}

\newcommand{\bZ}{\mathbb{Z}}
\newcommand{\bE}{\mathbb{E}}

\newcommand{\bfe}{\textbf{e}}

\newcommand{\dtone}{D_{\hat{T}_1}}
\newcommand{\dtn}{D_{\hat{T}_n}}

\newcommand{\Done}{\mathcal{D}_{\hat{T}_1}}
\newcommand{\Dn}{\mathcal{D}_{\hat{T}_n}}

\makeatother

\begin{document}
	
	\title{Isometric dilations for representations of product systems}
	
	\author[Barik]{Sibaprasad Barik}
	
	\address{Department of Mathematics, Technion-Israel Institute of Technology, 32000, Haifa, Israel}
\email{sibaprasadbarik00@gmail.com, sbarik@campus.technion.ac.il}

	\author[Bhattacharjee]{Monojit Bhattacharjee}
\address{Birla Institute of Technology and Science-Pilani, K. K. Birla Goa Campus, South Goa, 403726, India}

 \email{monojitb@goa.bits-pilani.ac.in, monojit.hcu@gmail.com}
	
	\author[Solel]{Baruch Solel}
	
	\address{Department of Mathematics, Technion-Israel Institute of Technology, 32000, Haifa, Israel}
\email{mabaruch@technion.ac.il}
	\subjclass[2020]{46L08, 46L10, 47A13, 47A20, 47A56, 47L05, 47L30, 47L55}
	
	\keywords{Completely contractive representation, $W^*$-correspondence, product system, isometric dilation, Szeg\"o positivity}

\begin{abstract}
	We discuss representations of product systems (of $W^*$-correspondences) over the semigroup $\bZ^n_+$ and show that, under certain pureness and Szego positivity conditions, a completely contractive representation can be dilated to an isometric representation. For $n=1,2$ this is known to hold in general (without assuming the conditions) but, for $n\geq 3$, it  does not hold in general (as is known for the special case of isometric dilations of a tuple of commuting contractions). Restricting to the case of tuples of commuting contractions, our result reduces to a result of Barik, Das, Haria and Sarkar.
	Our dilation is explicitly constructed and we present some applications.
\end{abstract}
	
	\maketitle


\section{ Introduction}

Consider a tuple of commuting contractions  $T=(T_1,\ldots,T_n)$ on a Hilbert space $H$. If $K$ is another Hilbert space that contains $H$ and $S=(S_1,\ldots,S_n)$ is a tuple of operators on K, we say that $S$ is a \emph{dilation} of $T$ if, for every non negative integers $k_1,k_2,\ldots,k_n$  and every $h\in H$,
\begin{equation}\label{defdilation} T_1^{k_1}T_2^{k_2}\cdots T_n^{k_n}h=P_H  S_1^{k_1}S_2^{k_2}\cdots S_n^{k_n}h \end{equation} 
where $P_H$ is the orthogonal projection of $K$ onto $H$.

Two special cases are when $H$ is invariant for each $S_i$ and then we call $S$ an \emph{extension} of $T$ and can write (\ref{defdilation}) simply as
$$T_ih=S_ih$$ for all $i$, or when $H$ is co-invariant for each $S_i$ (i.e. $K\ominus H$ is invariant) and then we call $S$ a \emph{coextension} of $T$. In the latter case, it is convenient to write $\imath$ for the inclusion of $H$ into $K$ (an isometry) and write (\ref{defdilation}) by
\[
\imath T_i^*=S_i^*\imath
\]for all $i$. In general, instead of just $K\supseteq H$ if there is an isometry $\Pi: H\to K$ such that
\begin{equation}\label{coextension}
	\Pi T_i^*=S_i^*\Pi
\end{equation}
for all $i$, then we say $S$ is a co-extension of $T$. We will mostly be interested in the case where $S_i$ commute and are isometries and refer to (\ref{coextension}) in this case as \emph{isometric dilation}. 

It is known that, if each $T_i$ is an isometry, $T$ has a unitary  extension $U=(U_1,\ldots, U_n)$ 	(so that each $U_i$ is unitary and the $U_i$ pairwise commute). Thus, if $T$ has an isometric dilation, it also has a unitary dilation. 

The existence of unitary or isometric dilation is important since it enables us, in many situations, to reduce the study of a tuple $T$ to that of a unitary, or an isometric, tuple which is better understood.

For $n=1$ or $n=2$ it is known (see \cite{NF} and \cite{An} ) that isometric dilation always exists. But Parrot (\cite{Pa}  ) showed that for $n\geq 3$ one can find commuting tuples that do not have an isometric dilation. 

This raises the problem of finding classes of commuting tuples for which isometric dilation exists.

In 1961 Brehmer introduced a condition on a commutative tuple of contractions that ensures that the tuple has an isometric dilation. In fact, this condition is equivalent to the existence of an isometric dilation with an extra property, called regular dilation (see \cite{Br} , \cite{NF} or \cite{AmMu} for details).

In \cite{GKVW} the authors introduced the class $\cP^n_{p,q}$ (where $1\leq p<q \leq n$) of $n$-commuting tuples of \emph{strict} contractions that satisfy the two positivity conditions
\begin{equation}\label{c1}
	\Sigma_{G\subseteq \{1,\ldots,n\}\backslash \{p\}} (-1)^{|G|} T_GT^{*}_G \geq 0 \end{equation}
and
\begin{equation}\label{c2}\Sigma_{G\subseteq \{1,\ldots,n\}\backslash \{q\}} (-1)^{|G|} T_GT^{*}_G \geq 0 \end{equation} 
	where $T_G=T_{i_1}\cdots T_{i_{k} } $ for $G=\{i_1,\ldots,i_k\}$. They proved that tuples in  $\cP^n_{p,q}$ have unitary dilations.

Another result in this direction can be found in \cite{BD22} where the class of tuples, considered by the authors for finding isometric dilation, enlarges the class considered in \cite{Br} and \cite{GKVW}. 

But the result that is most relevant to us here was obtained in \cite{BDHS19}. To state it we need the following notation. If $T=(T_1,\ldots,T_n)$ is a commuting tuple of contractions, we say that it is \emph{pure} if, for every $h\in H$ and $1\leq i \leq n$, $T_i^{*m}h\rightarrow 0$ as $m\rightarrow \infty$. Also, we write $\widehat{T_p}$ for the $n-1$ tuple which we get when we delete $T_p$ from $T$. Finally, we say that the tuple $T$ satisfies the \emph{Szego positivity condition} if
\begin{equation}\label{szego}
	\Sigma_{G\subseteq \{1,\ldots,n\}} (-1)^{|G|}T_GT_G^*  \geq 0.
\end{equation} 
Note that condition (\ref{c1}) means that $\widehat{T}_p$ satisfies the Szego condition (and similar comment applies to condition (\ref{c2})).

This allows us to define the class $\cT_{p,q}^n$ (for $1\leq p<q\leq n$) to be
\begin{equation}\label{Tpq}
	\cT_{p,q}^n:= \{T=(T_1,\ldots,T_n) : \widehat{T}_p   \mbox{   and   }   \widehat{T}_q \mbox{  satisfy  } (\ref{szego}) , \widehat{T}_q \mbox{   is pure  } \}.
\end{equation} 

The main result of \cite{BDHS19} is that every $T\in \cT_{p,q}^n$ has an isometric dilation. In fact such a dilation is explicitely constructed.

In this paper we study dilations of representations of product systems of $W^*$-correspondences. As we shall explain, a commuting tuple of contractions can be viewed as such a representation so that our study is a generalization of the study of dilations of such tuples.

We now fix a $W^*$-algebra $M$. A $W^*$-correspondence over the $W^*$-algebra $M$ is a
Hilbert $W^*$-module $E$ over $M$ endowed with the structure of a
left $M$-module via a normal  $^*$-homomorphism $\varphi_E :M \rightarrow
\mathcal{L}(E)$. An isomorphism between two correspondences is a bimodule isomorphism that preserves the inner product. Given two $W^*$-correspondences $E,F$ over $M$ one can define the (balanced) tensor product $E\otimes F$ and it is a $W^*$-correspondence with a natural left and right multiplication and the interior inner product given, on generators, by
$$\langle e_1\otimes f_1, e_2\otimes f_2 \rangle=\langle f_1, \varphi_E(\langle e_1,e_2\rangle_E)f_2\rangle $$ 
for $e_i\in E$ and $f_i\in F$.

If $M=\bC$, then a $W^*$-correspondence over $M$ is a Hilbert space and the tensor product has the usual meaning. However, note that, for a general $M$, $E\otimes F$ is not always isomorphic to $F\otimes E$.

A c.c. representation of $E$ on a
Hilbert space $H$ is a pair $(\sigma,T)$ where $\sigma$ is a
representation of $M$ on $H$ and $T:E\rightarrow B(H)$ is a
completely contractive linear map that is also a bimodule map
(that is, $T(a\cdot \xi \cdot b)=\sigma(a)T(\xi)\sigma(b)$ for
$a,b \in M$ and $\xi \in E$). The representation is said to be
isometric (or Toeplitz ) if $T(\xi)^*T(\eta)=\sigma(\langle
\xi,\eta \rangle)$ for every $\xi,\eta \in E$.

In \cite{Fo} Fowler studied product systems over a left-cancellative, countable, semigroup $P$ with an identity $e$. A product system $\bE$ is a family $\{X_s\}_{s\in P}$ with isomorphisms $\{\theta_{s,t}: X_s \otimes X_t \rightarrow X_{st} \}_{s,t \in P}$ such that $X_e=M$, the maps $\theta_{e,t}$ and $\theta_{s,e}$ are given by the left and right multiplications and it satisfies the associativity condition
$$\theta_{st,r}(\theta_{s,t}\otimes I_{X_r})=\theta_{s,tr}(I_{X_s}\otimes \theta_{t,r}) .$$

In this paper we study the case where $P=\bZ_+^n$. We write $\bfe_i=(0,0,\ldots, 1, \ldots,0)$ (with $1$ in the $ith$ position) and for a product system $\{X_s\}_{s\in \bZ_+^n}$, we write $E_i$ for $X_{\bfe_i}$, $1\leq i \leq n$.  

 It will be convenient to write $E_i^n$ for the $n$-fold
tensor product $E_i^{\otimes n}$ and to identify $X_{\textbf{m}}$ (for
$\textbf{m} =(m_1,m_2,\ldots,m_n)\in \bZ_+^n$) with $E_1^{m_1}\otimes E_2^{m_2} \otimes \cdots
\otimes E_k^{m_n}$ (where these tensor products are the balanced
tensor products over $M$). That means, in particular, that the
isomorphisms $\theta_{\bfe_i,\bfe_j}$, for $i\leq j$, are identity
maps. Setting $t_{i,j}=\theta_{\bfe_i,\bfe_j}:E_i\otimes E_j \rightarrow
E_j \otimes E_i$ for $i\geq j$ (and
$t_{i,j}=t_{j,i}^{-1}$ for $i<j$), one can check that the family
$\{t_{i,j}: 1\leq i,j \leq n\}$ satisfies
\be\label{assoct}
(t_{j,i}\otimes
I_{\bfe_l})(I_{\bfe_j}\otimes t_{l,i})(t_{l,j}\otimes I_{\bfe_i})=
(I_{\bfe_i}\otimes t_{l,j})
(t_{l,i}\otimes I_{\bfe_j})(I_{\bfe_l}\otimes t_{j,i})
\ee
for every $1\leq i,j,l \leq n$.
One can also check (but we omit the tedious computation) that,
given $n$ correspondences $E_1,\ldots,E_n$ over the $W^*$-algebra
$M$ and a family $\{t_{i,j}:1\leq i,j \leq n\}$ such that
$t_{i,j}:E_i \otimes E_j \rightarrow E_j\otimes E_i$ is an
isomorphism, $t_{i,j}=t_{j,i}^{-1}$ and $t_{i,i}$ is the identity
map, it determines, in a unique way, a product system $X$ (with
$X_{\textbf{m}}=E_1^{m_1}\otimes \cdots \otimes E_k^{m_k}$)
whose
isomorphisms 
$\{\theta_{\textbf{n},\textbf{m}}\}$ 
satisfy
$\theta_{\bfe_i,\bfe_j}=id$ 
if $i\leq j$ and
$\theta_{\bfe_i,\bfe_j}=
t_{i,j}$ if $i>j$.

In this paper we study representations of such product systems (and their isometric dilations). In Definition~\ref{repps} we define a c.c. (and isometric) representations of product systems and, in Example~\ref{commtuples} we show that, when $E_i=\bC$ for all $i$ and $t_{i,j}(x\otimes y)=y\otimes x$ (for $x\in E_i$ and $y\in E_j$), a c.c representation of the product system is given by a tuple of commuting contractions. Similarly, an isometric representation, in this case, is given by a tuple of commuting isometries.

In \cite[Theorem 3.3]{S08} it is shown that, under a condition that extends Brehmer's condition to product systems, the c.c, representation has an isometric dilation.

The main result of the current paper is Theorem~\ref{maindilation}. We prove there that, under conditions of pureness and Szego positivity  (as in (\ref{specialclass})), the c.c. representation of the product system has an isometric dilation. In fact, this dilation is explicitly constructed.

Note that, for $n=1$ or $n=2$ this result is known to hold in general (no conditions are needed). For $n=1$ it was proved in \cite{MS98} and for $n=2$ in \cite[Theorem 4.4]{S06}.  But, for $n\geq 3$, it  does not hold in general as is known for the special case of isometric dilations of a tuple of commuting contractions. 

Our main result is a generalization of the dilation result of \cite{BDHS19} . In the case where $E_i=E_j=\bC$, our result here, Theorem~\ref{maindilation}, reduces to the main result of \cite{BDHS19}.

In \cite{Fo} Fowler showed that there is a universal $C^*$- algebra, written $\cT_{\bE}$, which is universal for isometric representations of the product system $\bE$. It follows from our theorem that every c.c. representation of $\bE$ on $H$ that satisfies our conditions defines a completely positive map from $\cT_{\bE}$ into $B(H)$.

We present several examples to show that our result yields also some other interesting special cases.

As we noted above, every isometric tuple has an extension to a unitary tuple. This is not always the case for representations of a product system even for the case $n=1$ (see \cite[Chapter 5]{MS98} and \cite{MS06}).

\begin{section}{Preliminaries}\label{prel}
	We begin by recalling the notion of a $W^*$-correspondence.

\begin{definition}\label{corr}
	A $W^*$-correspondence over a $W^*$-algebra $M$ is a
	Hilbert $W^*$-module $E$ over $M$ endowed with the structure of a
	left $M$-module via a normal  $^*$-homomorphism $\varphi_E :M \rightarrow
	\mathcal{L}(E)$.
	
\end{definition}


If $E$ and $F$ are $W^*$-correspondences over $M$, then the
balanced tensor product $E\otimes_M F$ is a $W^*$-correspondence
over $M$. It is defined as the self dual Hausdorff completion of the
algebraic balanced tensor product with the internal inner product
given by
\be\label{tp}
\langle \xi_1 \otimes \eta_1,\xi_2\otimes \eta_2 \rangle=\langle
\eta_1, \varphi_F(\langle \xi_1,\xi_2 \rangle_E)\eta_2 \rangle_F
\ee
for all $\xi_1,\xi_2 \in E$ and $\eta_1,\eta_2 \in F$. The left
and right actions of $a\in M$ are defined by
\be\label{lrp}
\varphi_{E\otimes F}(a)(\xi \otimes \eta)b=\varphi_E(a)\xi \otimes
\eta b
\ee
for all $a,b\in M$, $\xi\in E$ and $\eta\in F$.

\begin{definition}\label{isomcor}
	An isomorphism of $W^*$-correspondences $E$ and $F$  is a
	surjective, bimodule map that preserves the inner products. We
	write $E\cong F$ if such an isomorphism exists.
\end{definition}

If $E$ is a $W^*$-correspondence over $M$ and $\sigma$ is a normal
representation of $M$ on a Hilbert space $H$  then
$E\otimes_{\sigma}H$ is the Hilbert space obtained as the
Hausdorff completion of the algebraic tensor product with respect
to $\langle \xi \otimes h, \eta \otimes k \rangle=\langle
h,\sigma(\langle \xi,\eta \rangle_E)k\rangle_H $. Given an
operator $X\in \mathcal{L}(E)$ and an operator $S\in \sigma(M)'$,
the map $\xi \otimes h \mapsto X\xi \otimes Sh $ defines a bounded
operator $X\otimes S$ on $E\otimes_{\sigma}H$. When $S=I_H$ and
$X=\varphi_E(a)$ (for $a\in A$) we get a normal representation of $M$ on
this Hilbert space.
We frequently
write $a\otimes I_H$ for $\varphi(a)\otimes I_H$.

\begin{definition}
	\label{Definition1.12}Let $E$ be a $W^*$-correspondence over a $W^*$-algebra
	$M$. Then a completely contractive covariant
	representation of $E$ (or, simply, a c.c. representation of $E$) on a
	Hilbert space $H$ is a pair $(\sigma,T)$, where
	
	\begin{enumerate}
		\item[(1)] $\sigma$ is a normal $\ast$-representation of $M$ in $B(H)$.
		\item[(2)] $T$ is a linear, completely contractive map from $E$ to
		$B(H)$.
		\item[(3)] $T$ is a bimodule map in the sense that $T(a\xi b)=\sigma(a)T(\xi
		)\sigma(b)$, $\xi\in E$, and $a,b\in M$.
	\end{enumerate}
	Such a representation is said to be isometric if, for every
	$\xi,\eta\in E$, $T(\xi)^*T(\eta)=\sigma(\langle \xi,\eta
	\rangle)$.

\end{definition}

As was shown in \cite[Lemmas 3.4--3.6]{MS98}, if a completely
contractive covariant representation,
$(\sigma,T)$, of $E$ in $B(H)$ is given, then it determines a contraction
$\tilde{T}:E\otimes_{\sigma}H\rightarrow H$ defined by the formula $\tilde
{T}(\eta\otimes h):=T(\eta)h$, $\eta\otimes h\in E\otimes_{\sigma}H$. The
operator $\tilde{T}$ satisfies
\begin{equation}
	\tilde{T}(\varphi(\cdot)\otimes I)=\sigma(\cdot)\tilde{T}.\label{covariance}%
\end{equation}
In fact we have the following lemma from \cite[Lemma 2.16]{MSNP}.

\begin{lemma}
	\label{CovRep}The map $(\sigma,T)\rightarrow\tilde{T}$ is a
	bijection between all completely contractive covariant
	representations $(\sigma,T)$ of $E$ on the Hilbert space $H$ and
	contractive operators $\tilde{T}:E\otimes_{\sigma }H\rightarrow H$
	that satisfy equation (\ref{covariance}). Given $\sigma$ and a
	contraction $\tilde{T}$ satisfying the covariance condition
	(\ref{covariance}), we get a  completely contractive covariant
	representation $(\sigma,T)$ of $E$ on $H$ by setting $T(\xi)h:=\tilde{T}%
	(\xi\otimes h)$.
	
	Moreover, the representation $(\sigma,T)$ is an isometric
	representation if and only if $\tilde{T}$ is an isometry.
\end{lemma}

\begin{remark}
	\label{GenPowers}In addition to $\tilde{T}$ we also require the
	\textquotedblleft generalized higher powers\textquotedblright\emph{\ }of
	$\tilde{T}$. These are\emph{\ }maps$\;\tilde{T}_{n}:E^{\otimes n}\otimes
	H\rightarrow H\;$defined by the equation$\;\tilde{T}_{n}(\xi_{1}\otimes
	\ldots\otimes\xi_{n}\otimes h)=T(\xi_{1})\cdots T(\xi_{n})h$, $\xi_{1}%
	\otimes\ldots\otimes\xi_{n}\otimes h\in E^{\otimes n}\otimes H$.
	One checks easily that
	$\tilde{T}_n=\tilde{T}\circ
	(I_E\otimes \tilde{T})\circ \cdots \circ (I_{E^{\otimes n-1}}
	\otimes \tilde{T})$, $n>1$.
	
\end{remark}

\begin{definition}
	A product system over $M$ is a family of $W^*$-correspondences $\{E_1,E_2. \ldots, E_n\}$ over $M$ together with unitary isomorphisms $t_{i,j}: E_i \otimes E_j \rightarrow E_j\otimes E_i$ (for $i>j$).	\end{definition} 
In fact, this defines a product system over the semigroup $\bZ_{+}^N$ and it is natural to define also $t_{i,i}=I_{E_i}$ and $t_{i,j}=t_{j,i}^{-1}$ for $i<j$. These unitaries allow us to identify $E_i\otimes E_j$ with $E_j\otimes E_i$. Similarly, we can identify $E_{n_1}\otimes E_{n_2} \otimes \cdots E_{n_k}$ with every permutation of the correspondences.

In this paper, we shall suppress the unitaries $t_{i,j}$ and write $E_i\otimes E_j=E_j\otimes E_i$.

\begin{definition} \label{repps}
	A c.c. (completely contractive) representation of a product system $(M,E_1,\ldots,E_n )$ is a tuple $(\sigma, T_1,T_2,\ldots,T_n)$ where for each $i$, $(\sigma, T_i)$ is a representation of $E_i$ on $H$ (and, thus, $\tilde{T}_i$ is a contraction) such that $T_i$ commutes with $T_j$ or, more precisely,
	\begin{equation}\label{commute} \tilde{T_i}(I_{E_i}\otimes \tilde{T_j})=\tilde{T_j}(I_{E_j}\otimes \tilde{T_i})(t_{i,j}\otimes I_H) 
		\end{equation}   for every $i,j$.
		If each $(\sigma, T_i)$ is an isometric representation of $E_i$ (equivalently, if each $\tilde{T_i}$ is an isometry ) , we say that the representation (of the product system) is an isometric representation.
	
	Suppressing $t_{i,j}$, we simply write $$ \tilde{T_i}(I_{E_i}\otimes \tilde{T_j})=\tilde{T_j}(I_{E_j}\otimes \tilde{T_i}).$$
\end{definition}

\begin{definition}\label{dilrepps}
We say that an isometric representation  $(\sigma, V_1,V_2,\ldots,V_n)$ of the product system 	$(M,E_1,\ldots,E_n )$ on the Hilbert space $K$ is a dilation of the representation  $(\sigma, T_1,T_2,\ldots,T_n)$ (on $H$) if there is an isometry $\Pi:H\rightarrow K$ such that 
$$ (I_{E_i} \otimes \Pi)\tilde{T_i}^*=\tilde{V_i}^* \Pi $$
for all $i$.

	\end{definition} 

We shall use the following notation.
For a subset $G=\{n_1,n_2,\ldots,n_k\}\subseteq \{1,2,\ldots,n\} $, we write $E_G$ or $E_{n_1,n_2,\ldots.n_k}$ for $E_{n_1}\otimes E_{n_2} \otimes \cdots \otimes E_{n_k}$.

Also, given a representation $(\sigma, T_1,T_2,\ldots,T_n)$ of the product system, and $G\subseteq \{1,2,\ldots,n\} $ as above, we write
\begin{equation}\label{TG} \tilde{T}_G = \tilde{T}_{n_1} (I_{n_1} \otimes \tilde{T}_{n_2}) \cdots 
(I_{G - \{n_k\}} \otimes \tilde{T}_{n_k})
: E_{n_1} \otimes \cdots \otimes E_{n_k} \otimes H =E_G\otimes H \to H.\end{equation}

and, for $1\leq k$,
\[ \tilde{T}_G^{(k)}= \tilde{T}_G(I_{E_G}\otimes \tilde{T}_G)\cdots (I_{E_G^{\otimes {k-1}}}\otimes \tilde{T}_G ).\]

Given a product system $(M,E_1,\ldots,E_n)$ and $\alpha\in \bZ_+^n$ we write $E^{\alpha}:=E_1^{\otimes \alpha_1}\otimes E_2^{\otimes \alpha_2}\otimes\cdots \otimes E_n^{\otimes \alpha_n}$ (with $\alpha^0:=M$) and form the associated Fock correspondence 
\begin{equation}\label{Fock}
	\cF(E):=\sum _{\alpha\in \bZ_+^n} E^{\alpha}.
	\end{equation} 

In the following we fix a $W^*$-algebra $M$, a product system $(M,E_1,\ldots,E_n)$ over $\bN^n$ and a representation $(\sigma,T_1,\ldots,T_n)$ of the product system over $H$.

We write $NRep(M)$ for the set of normal representations of $M$. For normal representations $\rho_1, \rho_2$ of $M$ we write $\rho_1\precsim \rho_2$ if $\rho_1$ is equivalent to a subrepresentation of $\rho_2$ and $\rho_1 \simeq \rho_2$ if they are equivalent. It will be convenient to write $G_1\precsim G_2$ if $G_i$ is the representation space of some normal representation of $M$, say $\rho_i$ and $\rho_1 \precsim \rho_2$.

In \cite{R1974} Rieffel defined the notion of a generator for $NRep(M)$ (based on a similar notion in pure algebra). In \cite[Proposition 1.1]{R1974} he showed that a normal representation $\rho$ is a generator if and only if every representation $\pi\in NRep(M)$ is a direct sum of copies of subrepresentations of $\rho$. In \cite[Proposition 1.3]{R1974} he proved that $\rho$ is a generator if and only if it is faithful.

As in \cite[Definition 4.1]{MS2013}, we say that a generator is an infinite generator if it is an infinite multiple of a generator for $NRep(M)$, i.e. it is an infinite multiple of a faithful normal representation of $M$. We can conclude that every infinite generator $\rho$ has the property that every normal rpresentation of $M$ is equivalent to a subrepresentation of $\rho$. We also note that every infinite generator is unique up to equivalence.

We shall now fix an infinite generator $\rho_{\infty}$ and write $H_{\infty} $ for its representation space.

The following lemma will be useful. The proof is straightforward and will be omitted.
\begin{lemma}\label{faithful}
	Suppose $E$ is a faithful $W^*$-correspondence over $M$ (that is, $\varphi_E$ is a faithful map) and $\rho$ is a faithful representation of $M$ on $H$. Then the representation $\varphi_E(\cdot)\otimes I_H$ is a faithful representation of $M$ on $E\otimes_{\rho}H$.
\end{lemma}

\end{section}
 
\begin{section}{Representations of product systems}

We will assume that $E_1,\ldots, E_n$ are faithful $W^*$-correspondences over a $W^*$-algebra $M$.
It follows from the lemma that, for $i=1,\ldots,n$, the representation $\varphi_i(\cdot)\otimes_{\rho_{\infty}}I_{H_{\infty}}$ is also an infinite generator.

Recall, for each $i=1,\ldots,n$,
$$\tilde{T_i}:E_i\otimes_{\sigma}H\rightarrow H$$
is defined by $\tilde{T_i}(\xi_i\otimes h)=T_i(\xi_i)h$ for $\xi_i\in E_i$ and $h\in H$. Recall that for any $G = \{n_1,\ldots,n_k\} \subseteq \{1,\ldots,n\}$, the operator $\tilde{T}_G$ is defined as in (\ref{TG}). 

We say, the representation $(\sigma,T_1,\ldots,T_n)$ satisfy Szeg\"o positivity if the tuple $\tilde{T}= (\tilde{T}_1,\ldots,\tilde{T}_n)$ is Szeg\"o positive, that is,  
\begin{equation} \label{Szego}
\mathbb{S}_n^{-1}\big(\tilde{T},\tilde{T}^*\big) := \sum_{G \subset \{1,\ldots,n\}} (-1)^{|G|} \tilde{T}_G \tilde{T}_G^* \geq 0.
\end{equation} 
If this holds, we write
\begin{equation}\label{defect} D_T^2=\mathbb{S}_n^{-1}\big(\tilde{T},\tilde{T}^*\big).\end{equation}

\begin{proposition}\label{pi}
	For a cc representation $(\sigma, T_1,\ldots,T_n)$ having the Sezg\"o positivity the map $\Pi: H \to \cF(E) \otimes H$ defined by 
	\[ \Pi h = \bigoplus_{\alpha \in \mathbb Z_+^n} (I_{E^{\alpha}} \otimes D_{\tilde{T}}) \tilde{T}^{*(\alpha)}h\]
	satisfies 
	\begin{equation}\label{normpi} \|\Pi h\|^2 =  \|h\|^2 +  \lim_{k \to \infty} \Big\langle \sum_{\emptyset \neq G \subseteq \{1,\ldots,n\}} (-1)^{|G|} \tilde{T}^k_G 
		\tilde{T}^{*k}_G  h, h \Big \rangle. \end{equation}   
	where $k\in\bZ_+$.
\end{proposition}

\begin{proof}   Using the fact that for each $i,j \in \{1,\ldots,n\}$, $\tilde{T}_i (I_{E_i} \otimes \tilde{T}_j) = \tilde{T}_j (I_{E_j} \otimes \tilde{T}_i)$, we have 
	\begin{align*}
		&\sum_{ \alpha = (\alpha_1,\ldots, \alpha_n) \in \mathbb Z_+^n } \langle \tilde{T}^{(\alpha)}(I_{E^{\alpha}} \otimes D^{2}_{\tilde{T}}) \tilde{T}^{*(\alpha)} h, h \rangle  \\
		&= \lim_{k \to \infty} \sum_{\text{max} \alpha_j \leq k-1} \langle \tilde{T}^{(\alpha)} (I_{E^{\alpha}} \otimes D^{2}_{\tilde{T}}) \tilde{T}^{*(\alpha)} h,h \rangle \\
		&= \lim_{k \to \infty} \sum_{\text{max} \alpha_j \leq k-1} \langle \tilde{T}^{(\alpha)} (I_{E^{\alpha}} \otimes \mathbb{S}_n^{-1}(\tilde{T}, \tilde{T}^*) )\tilde{T}^{*(\alpha)} h,h \rangle \\
		&= \lim_{k \to \infty} \sum_{\text{max} \alpha_j \leq k-1} \Big\langle \tilde{T}^{(\alpha)} \Big(I_{E^{\alpha}} \otimes \sum_{G \subseteq \{1,\ldots,n\}} (-1)^{|G|} \tilde{T}_G \tilde{T}^{*}_G\Big) \tilde{T}^{*(\alpha)} h,h \Big\rangle \\
		&= \lim_{k \to \infty} \sum_{\text{max} \beta_j \leq k} \Big\langle \tilde{T}^{(\beta)} \tilde{T}^{*(\beta)} \sum_{\alpha \leq \beta,\, \text{max} (\beta_j - \alpha_j)\leq 1,\, \text{max}\alpha_j \leq k-1} (-1)^{|\beta| - |\alpha|} h, h \Big\rangle.  
	\end{align*} 
	Note that, for a fixed $\beta$ with $\text{max} \beta_j \leq k$, \[
	\sum_{\alpha \leq \beta,\, \text{max}(\beta_j - \alpha_j)\leq 1,\, \text{max} \alpha_j \leq k-1} (-1)^{|\beta| - |\alpha|} = 0 \] unless $\beta = (\beta_1,\ldots, \beta_n) \in \{0,k\}^n$. If $\beta = (\beta_1,\ldots, \beta_n) \in \{0,k\}^n$ then it is equal to $(-1)^{|\text{supp} \beta|}$, where $|\text{supp} \beta|$ means the number of non-zero $\beta_j$'s in $\beta = (\beta_1,\ldots, \beta_n)$. Therefore, for every $h \in H$
	
	\begin{align*}
		\|\Pi h\|^2 &= \lim_{k \to \infty} \Big\langle \sum_{G \subseteq \{1,\ldots,n\}} (-1)^{|G|} \tilde{T}^{(k)}_G 
		\tilde{T}^{* (k)}_G  h, h \Big\rangle \\
		& = \|h\|^2 +  \lim_{k \to \infty} \Big\langle \sum_{\emptyset \neq G \subseteq \{1,\ldots,n\}} (-1)^{|G|} \tilde{T}^{(k)}_G 
		\tilde{T}^{* (k)}_G  h, h \Big \rangle
	\end{align*}
\end{proof}

A c.c. representation $(\sigma, T_1,\ldots, T_n)$ is said to be \textit{pure} if $(\sigma, T_i)$ is pure for each $i=1,\ldots, n$, that is, $\tilde{T}_i^{(k)}\tilde{T}_i^{*(k)}\to 0$ in WOT as $k\to\infty$. 

\begin{corollary}\label{piisom}
	Let $(\sigma, T_1,\ldots, T_n)$ be a pure c.c. representation which satisfies Szeg\"o positivity. Then the map defined in Proposition \ref{pi} is an isometry.
\end{corollary}
\begin{proof}
	Since $(\sigma, T_1,\ldots, T_n)$ is pure 
	\[
	\lim_{k \to \infty} \Big\langle \sum_{\emptyset \neq G \subseteq \{1,\ldots,n\}} (-1)^{|G|} \tilde{T}^{(k)}_G 
	\tilde{T}^{* (k)}_G  h, h \Big \rangle=0.
	\]
	So it follows from (\ref{normpi}) that $\Pi$ is an isometry.
\end{proof}

We denote, $\hat{T}_1:=(\sigma,T_2,\ldots,T_n)$ and $\hat{T}_n:=(\sigma,T_1,\ldots,T_{n-1})$ and with the help of them we define a class of c.c. representations namely $\cT_{1,n}^n(\sigma, M, E_1,\ldots, E_n, H)$ in the following way
\begin{equation}\label{specialclass}
\cT_{1,n}^n(\sigma, M, E_1,\ldots,E_n , H):=\{(\sigma, T_1,\ldots,T_n): \hat{T}_1,\hat{T}_n \text{ satisfy Szeg\"o positivity and }\hat{T}_n \text{ is pure}\}
\end{equation} 
\begin{example}\label{Cd} Let $M=\bC$ and $E_i=\bC^d$ ($i=1,\ldots,n$). Let $T_i:\bC^d\to\cB(H)$ be written as $T_i=(T_{i,1},\ldots,T_{i,d})$ where $T_{i,k}=T_i(e_k)$ ($\{e_k\} $ is the standard basis of $\bC^d$). Then, for each $i=1,\ldots,n$, $\tilde{T}_i: \bC^d\otimes H\to H$ is given by
\[
\tilde{T}_i(\lambda\otimes h)=T_i(\lambda)h=(T_{i,1},\ldots,T_{i,d})\begin{bmatrix}\lambda_1\\\vdots\\\lambda_d\end{bmatrix}=\sum_{j=1}^d\lambda_jT_i(e_j)h=\sum_{j=1}^d\lambda_jT_{i,j}h.
\]
To make it into a product system we need to fix an isomorphism $t_{i,j}:E_i\otimes E_j \rightarrow E_j\otimes E_i$. For simplicity we assume that 
$$t_{i,j}(e \otimes f)=f\otimes e$$
for $i\neq j$, $e\in E_i=\bC^d$ and $f\in E_j=\bC^d$ (but see the remark below).

For $i\neq j$, $h\in H$, $e_k\in E_i$ and $e_l\in E_j$, we apply Equation~\ref{commute} to get
$$T_i(e_k)T_j(e_l)h=\tilde{T_i}(I_{E_i}\otimes \tilde{T_j})(e_k\otimes e_l \otimes h)=\tilde{T_j}(I_{E_j}\otimes \tilde{T_i})(t_{i,j}\otimes I_H)(e_k\otimes e_l\otimes h)=$$  $$\tilde{T_j}(I_{E_j}\otimes \tilde{T_i})(e_l\otimes e_k \otimes h)=T_j(e_l)T_i(e_k)h.$$

Thus, we require that 
\begin{equation}\label{commtuple}T_{i,k}T_{j,l}=T_{j,l}T_{i,k}
	\end{equation}  for $i\neq j$ and $1\leq k,l \leq d$.

 Also, $\tilde{T}_i^{(2)}: \bC^d\otimes\bC^d\otimes H\to H$ is given by
\[
\tilde{T}_i(\lambda\otimes\mu\otimes h)=T_i(\lambda\otimes\sum_{k=1}^d\mu_kT_{i,k}h )=\sum_{j,k=1}^d\lambda_j\mu_kT_{i,j}T_{i,k}h
\]
Identifying $\bC^d\otimes\bC^d$ with $\bC^{d^2}$ and $\lambda\otimes\mu$ with the column $(\lambda_j\mu_k)_{j,k}$ we can write $T_i^{(2)}$ as  a row $(T_{i,j}T_{i,k})_{j,k}$. Similarly, for any $m\in\bN$, $T_i^{(m)}$ can be written as the row $(T_{i,j_1}\cdots T_{i,j_m})_{1\leq j_1,\ldots,j_m\leq d}$. Therefore,
\begin{equation}\label{Pureex}
T_i^{(m)}T_i^{(m)^*}=\sum_{1\leq j_1,\ldots,j_m\leq d}T_{i,j_1}\cdots T_{i,j_m}T_{i,j_m}^*\cdots T_{i,j_1}^*=\sum_{1\leq j_1,\ldots,j_m\leq d}\Big(\prod_{p=1}^m T_{i,j_p}\Big)\Big(\prod_{p=1}^m T_{i,j_p}^*\Big).
\end{equation}
In a similar fashion, for each $G=\{i_1,\ldots,i_k\}\subset\{1,\ldots,n\}$,
\[
\tilde{T}_G\tilde{T}_G^*=\sum_{1\leq j_1,\ldots,j_k\leq d}\Big(\prod_{p=1}^k T_{i_p,j_p}\Big)\Big(\prod_{p=1}^k T_{i_p,j_p}^*\Big). 
\]
Note that, due to the commutation relation (\ref{commtuple}), the last expression is independent of the order chosen on $G$.
Therefore, 
\begin{equation}\label{Szegoex}
\sum_{G=\{i_1,\ldots,i_k\}\subset\{1,\ldots,n\}}(-1)^{|G|}\tilde{T}_G\tilde{T}_G^*=\sum_{G=\{i_1,\ldots,i_k\}\subset\{1,\ldots,n\}}(-1)^{|G|}\sum_{1\leq j_1,\ldots,j_k\leq d}\Big(\prod_{p=1}^k T_{i_p,j_p}\Big)\Big(\prod_{p=1}^k T_{i_p,j_p}^*\Big).
\end{equation}
If the expression of the right hand side of (\ref{Pureex}) tends to $0$ in WOT as $m\to\infty$ then $T_i$ is pure and if it happens for each $i=1,\ldots,n$, then the c.c. representation $(\sigma,T_1,\ldots,T_n)$ is pure. On the other hand, if the expression on the right hand side of (\ref{Szegoex}) is a positive operator then the representation $(\sigma,T_1,\ldots,T_n)$ is Szeg\"o positive. If the above two hold together then the representation $(\sigma,T_1,\ldots,T_n)$ lies in the class $\cT_{1,n}^n(\sigma, \bC,\bC^d,\ldots,\bC^d , H)$.
\end{example} 

\

\begin{remark}\label{remarkcomm}
	In the example above, instead of choosing $t_{i,j}$ to be the flip, we can choose any other isomorphism of $\bC^d\otimes \bC^d$ (for $i>j$). Such an isomorphism is given by any unitary matrix of size $d^2\times d^2$. This will give us another, more involved, commutation relation (see, for example, \cite[Equation (1)]{PS}). We will not go into details here except to note that, if $d=1$ and, for every $i>j$ we fix some number $u_{i,j}$ with $|u_{i,j}|=1$ then the commutation relation is
	\begin{equation}\label{commuteu}
		T_iT_j=u_{i,j}T_jT_i
		\end{equation} for $i>j$.
		Note that, if $(T_1,\ldots, T_n)$ is a tuple of contractions that $u$-commute (in the sense of Equation (\ref{commuteu})) then $T_{G}$ may depend on the order on $G$ but $T_GT_{G}^*$ does not. In fact, in the expression of Equation (\ref{Szego}), $u_{i,j}$ will get canceled .
		
\end{remark}

\begin{example}\label{commtuples}
	If, in the setup of Remark~\ref{remarkcomm}, we take $d=1$ and $u_{i,j}=1$ for all $i,j$ then it reduces to the case of a commuting tuple of contractions studied by many researchers (see the introduction). Note that, in this case Definition\ref{dilrepps} reduces to Definition\ref{coextension}.
	
	\end{example}

\

Let $(\sigma, T_1,\ldots,T_n)\in\cT_{1,n}^n(\sigma, M, E_1,\ldots,E_n , H)$. Merging the $(n-1)$-tuples $\hat{T}_1$ and $\hat{T}_n$ we consider a new $(n-1)$-tuple $\hat{T}_{1n}$ in the following way:
\[
\tilde{\hat{T}}_{1n}=\big(\tilde{T}_1(I_1\otimes \tilde{T}_n),\tilde{T_2},\ldots,\tilde{T}_{n-1}\big).
\]
We have the following result.
\begin{proposition}\label{equality}
    If the tuples $\hat{T}_1$ and $\hat{T}_n$ satisfy Szeg\"o positivity then so is $\hat{T}_{1n}$. 
    Moreover, denoting 
    \[
    D_{\hat{T}_i}=\mathbb{S}_{n-1}^{-1}\big(\tilde{\hat{T}}_i,\tilde{\hat{T}}_i^*\big)^{1/2}\quad(\text{for all}\,\, i=1,2)\,\, \quad\text{and}\,\,\,D_{\hat{T}_{1n}}:=\mathbb{S}_{n-1}^{-1}\big(\tilde{\hat{T}}_{1n},\tilde{\hat{T}}_{1n}^*\big)^{1/2}
    \]
    we have
    \[D_{\hat{T}_{1n}}^2=\dtn ^2+\tilde{T}_1(I_1\otimes \dtone ^2)\tilde{T}_1^*=\dtone ^2+\tilde{T_n}(I_n\otimes \dtn ^2)\tilde{T}_n^*.\]
\end{proposition}
\begin{proof}
We observe that 
\begin{align*}
\mathbb{S}_{n-1}^{-1}&\big(\tilde{\hat{T}}_{1n},\tilde{\hat{T}}_{1n}^*\big)\\
&= \sum_{G \subset \{1,\ldots,n-1\}} (-1)^{|G|} \big(\tilde{\hat{T}}_{1n}\big)_G \big(\tilde{\hat{T}}_{1n}\big)_G^* \\
&= \sum_{G \subset \{2,\ldots,n-1\}} (-1)^{|G|} \tilde{T}_G \tilde{T}_G^*- \tilde{T}_1(I_1 \otimes \tilde{T}_n) \Big( \sum_{G \subset \{2,\ldots,n-1\}} (-1)^{|G|} \tilde{T}_G \tilde{T}_G^* \Big) (I_1 \otimes \tilde{T}_n^*) \tilde{T}_1^* \\
&= \sum_{G \subset \{2,\ldots,n-1\}} (-1)^{|G|} \tilde{T}_G \tilde{T}_G^* - \tilde{T}_1 \Big(  \sum_{G \subset \{2,\ldots,n-1\}} (-1)^{|G|} I_1\otimes \tilde{T}_G \tilde{T}_G^* \Big) \tilde{T}_1^* \\
& \hspace{4cm} + \tilde{T}_1 \Big(  \sum_{G \subset \{2,\ldots,n-1\}} (-1)^{|G|} I_1 \otimes \tilde{T}_G \tilde{T}_G^* \Big) \tilde{T}_1^{*} \\
& \hspace{4cm} -  \tilde{T}_1(I_1 \otimes \tilde{T}_n) \Big( \sum_{G \subset \{2,\ldots,n-1\}} (-1)^{|G|} I_{1n} \otimes \tilde{T}_G \tilde{T}_G^* \Big) (I_1 \otimes \tilde{T}_n^*) \tilde{T}_1^* \\
&=  \sum_{G \subset \{1,\ldots,n-1\}} (-1)^{|G|} \big(\tilde{\hat{T}}_{n}\big)_G \big(\tilde{\hat{T}}_{n}\big)_G^*   +  \tilde{T}_1 \Big(
I_1 \otimes \sum_{G \subset \{2,\ldots,n\}} (-1)^{|G|} \big(\tilde{\hat{T}}_{1}\big)_G \big(\tilde{\hat{T}}_{1}\big)_G^*  \Big) \tilde{T}_1^* \\
&= \mathbb{S}_{n-1}^{-1}\big(\tilde{\hat{T}}_{1},\tilde{\hat{T}}_{1}^*\big)  +   \tilde{T}_1 \Big( I_1 \otimes \mathbb{S}_{n-1}^{-1}\big(\tilde{\hat{T}}_{n},\tilde{\hat{T}}_{n}^*\big)  \Big) \tilde{T}_1^* \\
&= D_{\hat{T}_1}^2 + \tilde{T}_1\Big( I_1 \otimes D_{\hat{T}_n}^2 \Big) \tilde{T}_1^* .
\end{align*}
So, 
\[
\mathbb{S}_{n-1}^{-1}\big(\tilde{\hat{T}}_{1n},\tilde{\hat{T}}_{1n}^*\big)\geq 0
\]
 and hence
 \[
 D_{\hat{T}_{1n}}^2 =D_{\hat{T}_1}^2 + \tilde{T}_1\Big( I_1 \otimes D_{\hat{T}_n}^2 \Big) \tilde{T}_1^*.
 \]
Similarly, we can also show that 
\[ 
D_{\hat{T}_{1n}}^2 = D_{\hat{T}_n}^2 + \tilde{T}_n\Big( I_n \otimes D_{\hat{T}_1}^2 \Big) \tilde{T}_n^*.
\]
\end{proof}

From the above proposition it follows that for every $h\in H$,
\begin{equation}\label{eqnorm}
\|\dtn h\oplus (I_1\otimes \dtone )\tilde{T_1}^*h\|=\| \dtone h\oplus (I_n\otimes \dtn )\tilde{T_n}^*h\|.
\end{equation} 
Note that $\{ \dtn h\oplus (I_1\otimes \dtone )\tilde{T_1}^*h : h\in H\}$ is a subspace of $H\oplus (E_1\otimes H)$. On $H\oplus (E_1\otimes H)$ we have a representation $\tau_1$ of $M$ where, for $h,k\in H$ and $a\in M$, 
$$\tau_1(a)(h\oplus (\xi_1\otimes k))=\sigma(a)h\oplus (\varphi_1(a)\xi_1 \otimes k)$$
and $\tau_2$ is defined similarly on $H\oplus (E_n\otimes H)$.

Note that $D_{\tilde{T_i}}$ commutes with $\sigma(M)$ and, for $a\in M$, $\tilde{T_i}^*\sigma(a)=(\varphi_i(a)\otimes I_H)\tilde{T_i}^*$. It follows that
$$\tau_1(a)( \dtn h\oplus (I_1\otimes \dtone )\tilde{T_1}^*h)= \dtn \sigma(a)h\oplus (I_1\otimes \dtone )\tilde{T_1}^*\sigma(a)h$$
and, therefore, the space $\{ \dtn h\oplus (I_1\otimes \dtone )\tilde{T_1}^*h : h\in H\}$ is invariant under $\tau_1(M)$ . Writing $\cD_1$ for the closure
$$\cD_1=\{ \dtn h\oplus (I_1\otimes \dtone )\tilde{T_1}^*h : h\in H\}^-,$$ we see that, restricting $\tau_1$ to it, we get a subrepresentation of $\tau_1$.
Similarly, setting
$$\cD_2=\{ \dtone h\oplus (I_n\otimes \dtn)\tilde{T_n}^*h : h\in H\}^-$$ we get a subrepresentation of $\tau_2$.

It follows from (\ref{eqnorm}) that we can define a partial isometry (with initial space $\cD_1$) 
$$V_0:\cD_1 \rightarrow \cD_2$$
by $$V_0(\dtn h\oplus (I_1\otimes \dtone )\tilde{T_1}^*h)=\dtone h\oplus (I_n\otimes \dtn )\tilde{T_n}^*h.$$
From the above definition, it is straight-forward to see that $V_0$ intertwines $\tau_1$ and $\tau_2$. That is,
$$V_0 \tau_1(a)=\tau_2(a) V_0$$ 
for all $a \in M$ and we can write
$$\cD_1 \simeq \cD_2 .$$

\begin{proposition}\label{unitary}
	We can find representations $\rho_1, \rho_2$ of $ M$ on Hilbert spaces $\cE_1, \cE_2$ and a unitary operator $$U: \Done \oplus (E_n\otimes \Dn )\oplus \cE_2\rightarrow \Dn \oplus (E_1\otimes \Done )\oplus \cE_1   $$ that extends $U_0:=V_0^{-1}$ and intertwines the representations of $M$.	Moreover, \begin{enumerate}
		\item [(1)] There is a unitary operator $u_1:E_1 \otimes \cE_2 \rightarrow \cE_1$ (so that $u_1(E_1\otimes \cE_2)= \cE_1 $ ) that intertwines the representations $\rho_1$ and $\varphi_1(\cdot)\otimes_{\rho_2}I_{\cE_2}$.
		\item[(2)] There is a unitary operator $u_2: \cE_2 \rightarrow E_n \otimes\cE_1$ (so that $u_2\cE_2=E_n\otimes \cE_1 $) that intertwines the representations $\rho_2$ and $\varphi_2(\cdot)\otimes_{\rho_1}I_{\cE_1}$.
		\item[(3)] $\cE_i$ ($i=1,2$) are infinite generators of $NRep(M)$.
	\end{enumerate} 
\end{proposition}
\begin{proof}

We write $\cD_{\tilde{T_i}}$ for the closure of the range of $D_{\tilde{T_i}}$ to get $$\cD_1\subseteq \Dn  \oplus (E_1\otimes \Done ) $$ and
$$\cD_2\subseteq \Done  \oplus (E_n\otimes \Dn ) .$$ Note that the spaces on the right are also invariant under $\tau_1$ (or $\tau_2$) so they are (left) $M$-modules and so are

 $$\cM_2:=\Done  \oplus (E_n\otimes \Dn )\ominus \cD_2$$ and
$$\cM_1:=\Dn  \oplus (E_1\otimes \Done )\ominus \cD_1.$$ 
Now we write,
$$\cE_i :=H_{\infty} $$ (for $i=1,2$)  infinite generators. It follows that $E_n\otimes \cE_1$ and $E_1\otimes \cE_2$ are also infinite generators. The uniqueness of infinite generators implies the existence of $u_1,u_2$ as in the proposition.

Since $\cE_1\precsim \cM_i \oplus \cE_i$, $\cM_i \oplus\cE_i$ is also an infinite generator and it follows from the uniqueness that there is a unitary operator $u_1$ from $\cM_2\oplus \cE_2$ onto $\cM_1\oplus \cE_1$ that intertwines the representations. Setting $U=V_0^*\oplus u_1$ completes the proof.

\end{proof}

To simplify the notation we write $E_{ij}$ for $E_i\otimes E_j$ and identify it with $E_j\otimes E_i$. Similarly, we write $E_{ijk}$ for $E_i\otimes E_j\otimes E_k$ etc.

We have from the Proposition \ref{equality} that
$$D_{\hat{T}_{1n}}^2=\dtn ^2+\tilde{T_1}(I_1\otimes \dtone ^2)\tilde{T}_1^*=\dtone ^2+\tilde{T_n}(I_n\otimes \dtn ^2)\tilde{T}_n^*\geq 0.$$
We, thus, get an isometry $V:\cD_{\hat{T}_{1n}}\rightarrow \Dn \oplus (E_1\otimes \Done )\oplus \cE_1 $ such that 
$$V(D_{\hat{T}_{1n}}h)=\dtn h \oplus (I_1\otimes \dtone )\tilde{T}_1^*h\oplus 0_{\cE_1}.$$

We now need the following notation:
\begin{enumerate} 
\item[(1)] $P_1$ is the projection $$P_1:\Dn  \oplus (E_1\otimes \Done )\oplus \cE_1 \rightarrow 	E_1\otimes \Done .$$
\item[(2)] $i_2'$ is the inclusion $$i_2':(E_{1n} \otimes \Dn ) \oplus (E_1 \otimes \cE_2) \rightarrow (E_1\otimes \Done )\oplus (E_{1n} \otimes \Dn ) \oplus (E_1 \otimes \cE_2) .$$
\item[(3)] $j_2'=i_2'\circ (I_{E_{1n}\otimes \cD_{\tilde{T}_1}}\oplus (I_1\otimes u_2^*)): (E_{1n}\otimes \cD_{\tilde{T}_1})\oplus (E_{1n}\otimes \cE_1)\rightarrow (E_{1n}\otimes \cD_{\tilde{T}_1})\oplus (E_1\otimes \cE_2)\oplus (E_1\otimes \Done )$.
\item[(4)] $i_2$ is the inclusion 
$$i_2:\Dn  \oplus \cE_1 \rightarrow \Dn \oplus (E_1 \otimes \Done ) \oplus \cE_1$$ so that $i_2^*$ is the projection $$i_2^*:\Dn \oplus (E_1 \otimes \Done ) \oplus \cE_1 \rightarrow \Dn  \oplus \cE_1 .$$
	
	\end{enumerate} 

It will be convenient to write $I_1, I_n, I_{1n}$ for $I_{E_1}, I_{E_n}, I_{E_{1n}}$ respectively.

Note that, since $U$ intertwines the representations of $M$, the operator $I_1 \otimes U$ is well defined and bounded. So the following operators are well defined
$$(I_1\otimes U)P_1:\Dn  \oplus (E_1\otimes \Done )\oplus \cE_1 \rightarrow (E_1\otimes \Dn )\oplus (E_1 \otimes E_1\otimes \Done )\oplus E_1\otimes \cE_1    $$
and
$$(I_1\otimes U)j_2':(E_{1n} \otimes \Dn ) \oplus (E_{1n} \otimes \cE_1)\rightarrow (E_1\otimes \Dn )\oplus (E_1\otimes E_1 \otimes \Done )\oplus( E_1\otimes \cE_1 ).$$

We can now form the operator
\begin{equation}\label{defU1} U_1:= \left(\begin{array}{cc} (I_1\otimes U)P_1 & (I_1\otimes U)j_2' \\ i_2^* & 0 \end{array}\right) : \left(\begin{array}{c} \Dn  \oplus (E_1\otimes \Done )\oplus \cE_1 \\ (E_{1n} \otimes \Dn ) \oplus (E_{1n} \otimes \cE_1) \end{array}\right) \end{equation}   $$\rightarrow \left(\begin{array}{c}  (E_1\otimes \Dn )\oplus (E_1\otimes E_1 \otimes \Done )\oplus ( E_1\otimes \cE_1 )\\ \Dn  \oplus \cE_1 \end{array}\right). $$ 

A straightforward computation shows that, as $U$ is unitary, then so is $U_1$.

To shorten some computations, we shall write 
\begin{equation}\label{D} \mathcal{D}:=\Dn \oplus (E_1\otimes \Done )\oplus \cE_1
	\end{equation} 
and
\begin{equation}\label{Dprime} \cD':=\Dn  \oplus \cE_1
	\end{equation} 
so that
$$U_1:= \left(\begin{array}{cc} (I_1\otimes U)P_1 & (I_1\otimes U)j_2' \\ i_2^* & 0 \end{array}\right) : \left(\begin{array}{c} \cD  \\ E_{1n} \otimes \cD' \end{array}\right) \rightarrow \left(\begin{array}{c}  E_1\otimes \cD\\ \cD' \end{array}\right). $$

Note that $M$ acts on  $\mathcal{D}$ by the restriction, to $\mathcal{D}$, of $\sigma \oplus ( \varphi_1 \otimes I_{\Done})\oplus \rho_1$ we write $\rho_{\cD}$ for this representation. Similarly, $\rho_{\cD'}$ will denote the  representation of $M$ on $\cD'$ (the restriction of $\sigma \oplus \rho_1$). There is a natural left action on $E_{1n}=E_1\otimes E_n$ which comes through an adjointable map $\varphi_{1n}: M\to \cL(E_1\otimes E_n)$ which is defined by
\begin{equation}\label{action1}
	\varphi_{1n}(a)(e_1\otimes e_n):=\varphi_1(a)e_1\otimes e_n\quad (\text{for all}\,\, a\in M,\, e_1\in E_1, \,e_n\in E_n)
\end{equation}
Viewing $E_{1n}$ as $E_n\otimes E_1$ the action $\varphi_{1n}$ can also be represented by,
\begin{equation}\label{action2}
	\varphi_{1n}(a)(e_n\otimes e_1)=\varphi_n(a)e_n\otimes e_1\quad (\text{for all}\,\, a\in M,\, e_1\in E_1, \,e_n\in E_n).
\end{equation}
We have the following lemma.
\begin{lemma}\label{Uintertwines}
	The operator $U_1$ intertwines the representations of $M$, that is
	$$U_1 (\rho_{\cD}(a) \oplus (\varphi_{1n}(a)\otimes I_{\cD'}))=((\varphi_1(a) \otimes I_{\cD})\oplus \rho_{\cD'}(a))U_1$$ for every $a\in M$.
\end{lemma}
\begin{proof} 
	We need the prove the following, for every $a\in M$,
	\begin{enumerate}
		\item [(i)] $(I_1\otimes U)P_1 \rho_{\cD}(a)=((\varphi_1(a)\otimes I_{\cD})(I_1\otimes U)P_1 $.
		\item[(ii)]  $((I_1\otimes U)j_2') (\varphi_{1n}(a)\otimes I_{\cD'})=(\varphi_1(a)\otimes I_{\cD})((I_1\otimes U)j_2')$.
		\item[(iii)] $i_2^* \rho_{\cD}(a)=\rho_{\cD'}(a)i_2^*$.
	\end{enumerate}
	Part (iii) is obvious. We proceed to prove part (i). From the definitions of $P_1$ and of $\rho_{\cD}$ it follows that 
	$$P_1\rho_{\cD}(a)=(\varphi_1(a) \otimes I_{\Done})P_1.$$
	It follows from Proposition~\ref{unitary}(4) that $U\sigma|\cD_{\tilde{T}_2}=\rho_{\cD}U|\cD_{\tilde{T}_2}$.
	Using this, we have
	$$(I_1\otimes U)P_1\rho_{\cD}(a)=(I_1\otimes U)(\varphi_1(a) \otimes I_{\Done})P_1=(\varphi_1(a) \otimes I_{\cD})(I_1\otimes U)P_1 ,$$ 
	proving (i).
	For (ii) we write
	$$((I_1\otimes U)j_2') (\varphi_{1n}(a)\otimes I_{\cD'})=(I_1\otimes U)i_2'(I_{E_{1n}\otimes \Dn}\oplus (I_1\otimes u_2^*))(\varphi_{1n}(a)\otimes I_{\cD'})$$  $$=(I_1\otimes U)i_2'((\varphi_{1n}(a)\otimes I_{\Dn})\oplus (\varphi_1(a)\otimes I_{\cE_2}))(I_{E_{1n}\otimes \Dn}\oplus (I_1\otimes u_2^*))$$  $$=(I_1\otimes U)((\varphi_{1n}(a)\otimes I_{\Dn})\oplus (\varphi_1(a)\otimes I_{\cE_2})\oplus (\varphi_1(a)\otimes I_{\Done}))i_2'(I_{E_{1n}\otimes \Dn}\oplus (I_1\otimes u_2^*)).$$  Using the intertwining property of $U$ again, we see that this is equal to
	$$(\varphi_1(a) \otimes I_{\cD})(I_1\otimes U)i_2'(I_{E_{1n}\otimes \Dn}\oplus (I_1\otimes u_2^*))=(\varphi_1(a)\otimes I_{\cD})((I_1\otimes U)j_2')$$ proving (ii).
\end{proof}

\begin{lemma}\label{U1}
For $h\in H$,
$$U_1(VD_{\hat{T}_{1n}}h, (I_{1n}\otimes \dtn )(I_1\otimes \tilde{T}_n^*)\tilde{T}_1^*h\oplus 0_{E_{1n}\otimes \cE_1})=((I_1\otimes VD_{\hat{T}_{1n}})\tilde{T}_1^*h, \dtn h\oplus 0_{\cE_1}).$$
\end{lemma}
\begin{proof}
We start by proving that, for $f\in E_1\otimes H$, 
\begin{equation}\label{f}
(I_1\otimes U)[(I_1\otimes \dtone )f+ (I_{1n}\otimes \dtn )(I_1\otimes \tilde{T}_n^*)f]=(I_1\otimes \dtn )f+(I_{11}\otimes \dtone )(I_1\otimes \tilde{T}_1^*)f.
\end{equation}
Clearly, it suffices to prove it for $f=\xi \otimes g$ where $\xi\in E_1$ and $g\in H$. So we compute
\begin{align*}
&(I_1\otimes U)[(I_1\otimes \dtone )(\xi\otimes g)\oplus (I_{1n}\otimes \dtn )(I_1\otimes \tilde{T}_n^*)(\xi\otimes g)\oplus 0_{E_1\oplus \cE_2}]\\
=& (I_1\otimes U)[\xi \otimes \dtone g \oplus(I_{1n}\otimes \dtn )(\xi\otimes \tilde{T}_n^*g)\oplus 0_{E_1\oplus \cE_2}]\\
=& (I_1\otimes U)[\xi \otimes \dtone g \oplus \xi \otimes (I_n\otimes \dtn )\tilde{T}_n^*g \oplus 0_{E_1\oplus \cE_2}]\\
=&\xi \otimes U(\dtone g \oplus (I_n\otimes \dtn )\tilde{T}_n^*g \oplus 0_{\cE_2})
\end{align*}
Since $U$ extends $U_0 $ ($=V_0^{-1}$), $U(\dtone g \oplus (I_n\otimes \dtn )\tilde{T}_n^*g \oplus 0_{\cE_2})=(\dtn g\oplus (I_1\otimes \dtone )\tilde{T}_1^*g \oplus 0_{\cE_1})$ and, thus,
\begin{align*}
&(I_1\otimes U)[(I_1\otimes \dtone )(\xi\otimes g)\oplus (I_{1n}\otimes \dtn )(I_1\otimes \tilde{T}_n^*)(\xi\otimes g)\oplus 0_{E_1\oplus \cE_2}]\\
=&\xi \otimes( \dtn g\oplus (I_1\otimes \dtone )\tilde{T}_1^*g\oplus 0_{\cE_1})\\
=&(I_1\otimes \dtn )(\xi\otimes g)\oplus (I_{11}\otimes \dtone )(I_1\otimes \tilde{T}_1^*)(\xi\otimes g) \oplus 0_{E_1\oplus \cE_1}
\end{align*}
proving (\ref{f}).

Now we compute
\begin{align*}
&U_1(VD_{\hat{T}_{1n}}h, (I_{1n}\otimes D_{\tilde{T}_1})(I_1\otimes \tilde{T}_n^*)\tilde{T}_1^*h\oplus 0_{E_{1n}\otimes \cE_1})\\
=& U_1(\dtn h\oplus (I_1\otimes \dtone )\tilde{T}_1^*h\oplus 0_{\cE_1}, (I_{1n}\otimes \dtn )(I_n\otimes \tilde{T}_1^*)\tilde{T_n}^*h\oplus 0_{E_{1n}\otimes \cE_1})\\
=& ((I_1\otimes U)[(I_1\otimes D_{\tilde{T_n}})\tilde{T}_1^*h\oplus (I_{1n}\otimes \dtn )(I_1\otimes \tilde{T}_n^*)\tilde{T}_1^*h\oplus 0_{E_1\oplus \cE_2}], \dtn h\oplus 0_{\cE_1})\\ 
=&((I_1\otimes \dtn )\tilde{T}_1^*h\oplus (I_{11}\otimes \dtone )(I_1\otimes \tilde{T}_1^*)\tilde{T}_1^*h \oplus 0_{E_1\oplus \cE_1}, \dtn h\ \oplus 0_{\cE_1})
\end{align*}
where we used equation (\ref{f}) with $f=\tilde{T}_1^*h$. It is left to show that
$$(I_1\otimes \dtn )\tilde{T}_1^*h\oplus (I_{11}\otimes \dtone )(I_1\otimes \tilde{T}_1^*)\tilde{T}_1^*h \oplus 0_{E_1\oplus \cE_1}=(I_1\otimes VD_{\hat{T}_{1n}})\tilde{T}_1^*h.$$ 
Now, replace $\tilde{T}_1^*h$ by $\xi \otimes g$ and use the definition of $V$, that completes the proof.
\end{proof}


For faithful $W^*$-correspondences $E_1,\ldots, E_n$ over a $W^*$-algebra $M$, we consider a new family of $n-1$ faithful correspondences $E_{1n}, E_2,\ldots ,E_{n-1}$. Now the Fock correspondence for this product system (constructed as in (\ref{Fock})) is 
\begin{equation}\label{FE} 
\cF(E) = \bigoplus_{\alpha \in \bZ_+^{n-1}}  E^{\alpha} 
\end{equation}  where $E^0 = M$ and for each $0 \neq \alpha = (\alpha_1,\ldots,\alpha_{n-1}) \in \bZ_+^{n-1}$, $E^{\alpha} = E_{1n}^{\otimes \alpha_1} \otimes E_2^{\otimes \alpha_2} \otimes \cdots \otimes E_{n-1}^{\otimes \alpha_{n-1}}$. 
For each $i \in \{2,\ldots,n-1\}$, we consider the left creation operator $L_i: E_i \to \cB(\cF(E)\otimes H)$  defined by 
\[
L_i(\xi_i)=L_{\xi_i}  \quad \quad \text{for}\,\, \xi_i \in E_i,\,\, \text{and}\,\, i=2,\ldots,n-1
\]
and $L_1: E_{1n} \to \cB(\cF(E)\otimes H)$ by $L_1(\xi_{1n})= L_{\xi_{1n}}$ 
The operators $L_{\xi_i}: \cF(E)\otimes H \to \cF(E)\otimes H$ is defined by
\[
L_{\xi_i}(h)=\xi_i \otimes h\quad \quad \text{and}\quad L_{\xi_i}(\eta^{(\alpha)}\otimes h)=\xi_i \otimes \eta^{(\alpha)}\otimes h\quad (\eta^{(\alpha)}\in E^{\alpha}, h\in H).
\] The operator $L_{\xi_{1n}}:\cF(E)\otimes H \to \cF(E)\otimes H$ is defined by $L_{\xi_{1n}}(\eta^{(\alpha)} \otimes h) = \xi_{1n} \otimes \eta^{(\alpha)} \otimes h$.
Following this, for each $i=2,\ldots,n-1$, we define 
$\tilde{L}_i :E_i \otimes\cF(E)\otimes H\to \cF(E)\otimes H$ by
\be \label{Ltilde}
\tilde{L}_i(\xi_i\otimes k):=L_i(\xi_i)k=\xi_i \otimes k\quad \text{and}  \quad 
\tilde{L}_1(\xi_{1n} \otimes k):= L_1(\xi_{1n})k =\xi_{1n} \otimes k \quad (k\in \cF(E)\otimes H).
\ee
For the above defined $(n-1)$-tuple $\tilde{\hat{T}}_{1n}=\big(\tilde{T}_1(I_1\otimes \tilde{T}_n),\tilde{T_2},\ldots,\tilde{T}_{n-1}\big)$, we consider a dilation map $\Pi_{1n}:H\to \cF(E)\otimes H$, that is defined by
\begin{equation}\label{pi2}
\Pi_{1n} h:= \bigoplus_{\alpha \in \mathbb Z_+^{n-1}} (I_{E^{\alpha}}\otimes D_{\hat{T}_{1n}})\tilde{T}_{1n}^{*(\alpha)}h = D_{\hat{T}_{1n}}h\oplus(I_{E_{1n}}\otimes D_{\hat{T}_{1n}})\tilde{T}_{1n}^{*}h\oplus(I_{E_2}\otimes D_{\hat{T}_{1n}})\tilde{T}_2^{*}h\oplus\cdots
\end{equation}
 where for $\alpha = (\alpha_1,\ldots,\alpha_{n-1}) \in \mathbb Z_+^{n-1}$, $E^{\alpha}= E_{1n}^{\otimes \alpha_1} \otimes \cdots \otimes E_{n-1}^{\otimes \alpha_{n-1}}$ and  $\tilde{T}_{1n}^{(\alpha)} = \tilde{T}_{1n}^{\alpha_1} \tilde{T}_2^{\alpha_2} \cdots \tilde{T}_{n-1}^{\alpha_{n-1} - 1} $.

\begin{remark}  
 Note that this is the same map as in Proposition~\ref{pi} and in Corollary~\ref{piisom} but for the product system $(M, E_{1n},E_2, \ldots, E_{n-1})$ and the tuple $\tilde{\hat{T}}_{1n}=(\tilde{T}_1(I_1\otimes \tilde{T}_n),\tilde{T}_2,\ldots, \tilde{T}_{n-1})$.
 \end{remark} 



Now, we are ready to prove a dilation result for the $(n-1)$-tuple $\tilde{\hat{T}}_{1n}$.
 \begin{proposition}\label{dilation} Let $\hat{T}_n$ be a pure $(n-1)$-tuple. Then the above defined map $\Pi_{1n}$ is an isometry and 
 \begin{enumerate}
 \item[(a)] $(I_{1n}\otimes \Pi_{1n})\tilde{T}_{1n}^*=\tilde{L}_1^*\Pi_{1n}$.
 \item[(b)] $(I_i\otimes \Pi_{1n})\tilde{T}_i^*=\tilde{L}_i^*\Pi_{1n}$ for all $i=2,\ldots, n-1$.
 \end{enumerate}
 \end{proposition}
\begin{proof} The fact that $\Pi_{1n}$ is an isometry follows from Corollary~\ref{piisom} (see the remark above). 
Let $\xi_{1n}\in E_{1n},\eta^{(\beta)}\in E^{\beta}$ and $h\in H$. Then we have
\begingroup
	\allowdisplaybreaks
 \begin{align*}
 &\langle\tilde{L}_1^*\Pi_{1n} h, \xi_{1n} \otimes\eta^{(\beta)}\otimes h \rangle \\
 = &\langle\Pi_{1n} h, \tilde{L}_1(\xi_{1n} \otimes \eta^{(\beta)}\otimes h) \rangle \\
 =&\Big\langle \bigoplus_{\alpha \in \mathbb Z_+^{n-1}} (I_{E^{\alpha}}\otimes D_{\hat{T}_{1n}})\tilde{T}_{1n}^{*(\alpha)}h, \xi_{1n} \otimes\eta^{(\beta)}\otimes h\Big\rangle\\
 =& \langle (I_{E^{\beta}} \otimes I_{1n} \otimes D_{\hat{T}_{1n}})\tilde{T}_{1n}^{*(\beta + e_1)}h, \xi_{1n} \otimes\eta^{(\beta)}\otimes h\rangle\\
 =& \langle[I_{1n} \otimes(I_{E^{\beta}} \otimes D_{\hat{T}_{1n}})\tilde{T}_{1n}^{*(\beta)}]\tilde{T}_{1n}^*h, \xi_{1n} \otimes\eta^{(\beta)}\otimes h\rangle \\
 =&\langle (I_{1n}\otimes \Pi_{1n}) \tilde{T}_{1n}^*, \xi_{1n} \otimes \eta^{(\beta)} \otimes h \rangle .
 \end{align*}
 \endgroup
On the other hand, for each $i=2,\ldots,n-1$, take $\xi_i\in E_i,\eta^{(\beta)}\in E^{\beta}$ and $h\in H$. Then we have
\begingroup
	\allowdisplaybreaks
 \begin{align*}
 &\langle\tilde{L}_i^*\Pi_{1n} h, \xi_i \otimes\eta^{(\beta)}\otimes h \rangle \\
 = &\langle\Pi_{1n} h, \tilde{L}_i(\xi_i \otimes \eta^{(\beta)}\otimes h) \rangle \\
 =&\Big\langle \bigoplus_{\alpha \in \mathbb Z_+^{n-1}} (I_{E^{\alpha}}\otimes D_{\hat{T}_{1n}})\tilde{T}_{1n}^{*(\alpha)}h, \xi_i \otimes\eta^{(\beta)}\otimes h\Big\rangle\\
 =& \langle (I_{E^{\beta+e_i}} \otimes D_{\hat{T}_{1n}})\tilde{T}_{1n}^{*(\beta + e_i)}h, \xi_i \otimes\eta^{(\beta)}\otimes h\rangle\\
 =& \langle[I_i \otimes(I_{E^{\beta}} \otimes D_{\hat{T}_{1n}})\tilde{T}_{1n}^{*(\beta)}]\tilde{T}_{i}^*h, \xi_i \otimes\eta^{(\beta)}\otimes h\rangle \\
 =&\langle (I_{i}\otimes \Pi_{1n}) \tilde{T}_{i}^*, \xi_i \otimes \eta^{(\beta)} \otimes h \rangle .
 \end{align*}
 \endgroup
\end{proof}

Recall that $V:\cD_{\hat{T}_{1n}}\rightarrow \Dn \oplus (E_1\otimes \Done )\oplus \cE_1 =\cD$ is an isometry such that 
$$V(D_{\hat{T}_{1n}}h)=\dtn h \oplus (I_1\otimes \dtone )\tilde{T_1}^*h\oplus 0_{\cE_1}.$$

We now define the map
$$\Pi_{1n,V}:H\rightarrow \cF(E)\otimes [\cD \oplus (E_{1n}\otimes \cD)]$$
by
$$\Pi_{1n,V}h=(I_{\cF(E)}\otimes V)\Pi h= \bigoplus_{\alpha \in \bZ_+^{n-1}}\big(I_{E^{(\alpha)}}\otimes VD_{\hat{T}_{1n}}\big)\tilde{T}_{1n}^{*(\alpha)}h.$$
Since both $\Pi$ and $V$ are isometries, so is $\Pi_{1n,V}$. Therefore, we have the following dilation result which basically replace dilation space by some other Fock space where coefficient space is $\cD$ instead of $\cD_{\hat{T}_{1n}}$.
 
 \begin{proposition}\label{dilationV} Let $\hat{T}_n$ be a pure $(n-1)$-tuple and the map $\Pi_{1n,V}$ as above. Then, 
	\begin{enumerate}
		\item[(a)] $(I_{1n}\otimes \Pi_{1n,V})\tilde{T}_{1n}^*=\tilde{L}_1^*\Pi_{1n,V}$.
		\item[(b)] $(I_i\otimes \Pi_{1n,V})\tilde{T}_i^*=\tilde{L}_i^*\Pi_{1n,V}$ for all $i=2,\ldots, n-1$.
	\end{enumerate}
\end{proposition}
\begin{proof}
	To prove,
\begin{equation}\label{VL}	(I_{1n}\otimes I_{\cF(E)}\otimes V)\tilde{L}_{1}^*=\tilde{L}_{1}^*(I_{\cF(E)}\otimes V),
\end{equation}
note that $\tilde{L}_1^*$ (on the left hand side) is the inclusion of $E_{1n}\otimes \cF(E) \otimes \cD_{\hat{T}_{1n}}$ in $\cF(E) \otimes \cD_{\hat{T}_{1n}}$ while $\tilde{L}_1^*$ (on the right hand side) is the inclusion of $E_{1n}\otimes \cF(E) \otimes \cD$ in $\cF(E) \otimes \cD$. Thus, for all $h\in \cD_{\hat{T}_{1n}}$, $ \xi_{1n} \in E_{1n}$  and $\eta^{(\beta)}\in E^{\beta}$,
$$(I_{1n}\otimes I_{\cF(E)}\otimes V)\tilde{L}_{1}^*(\xi_{1n}\otimes \eta^{(\beta)}\otimes h)=\xi_{1n}\otimes \eta^{(\beta)}\otimes Vh=(\xi_{1n}\otimes \eta^{(\beta)})\otimes Vh$$  $$=\tilde{L}_{1}^*(I_{\cF(E)}\otimes V)(\xi_{1n}\otimes \eta^{(\beta)}\otimes h),$$ proving (\ref{VL}).
Using this,
\begin{align*}
(I_{1n}\otimes \Pi_{1n,V})\tilde{T}_{1n}^*&=\big(I_{1n}\otimes (I_{\cF(E)}\otimes V)\Pi_{1n}\big)\tilde{T}_{1n}^*\\&=(I_{1n}\otimes I_{\cF(E)}\otimes V)(I_{1n}\otimes \Pi_{1n} )\tilde{T}_{1n}^*\\&=(I_{1n}\otimes I_{\cF(E)}\otimes V)\tilde{L}_1^*\Pi_{1n}\qquad(\text{by (a) of Proposition \ref{dilation}})\\
&=\tilde{L}_{1}^*(I_{\cF(E)}\otimes V)\Pi_{1n}\qquad(\text{by \ref{VL}})\\
&=\tilde{L}_{1}^*\Pi_{1n,V}.
\end{align*}
This proves (a). Proof of item (b) is similar.
\end{proof}	 \
 
 Let, $U_1=\left(\begin{array}{cc} A & B \\ C & 0 \end{array}\right)$. So, by the Lemma \ref{U1}, \[\left(\begin{array}{cc} A & B \\ C & 0 \end{array}\right) \left(\begin{array}{c} VD_{\hat{T}_{1n}}h \\ (I_{1n}\otimes \dtn )(I_1\otimes \tilde{T_n}^*)\tilde{T}_1^*h \oplus 0_{E_{1n}\otimes \cE_1}\end{array}\right)=\left(\begin{array}{c}  (I_1\otimes VD_{\hat{T}_{1n}})\tilde{T}_1^*h\\ \dtn h\oplus 0_{\cE_1}\end{array}\right).\]

 This implies,
 
 \[
 AVD_{\hat{T}_{1n}}h+B\big[(I_{1n}\otimes \dtn )(I_1\otimes \tilde{T_n}^*)\tilde{T}_1^*h \oplus 0_{E_{1n}\otimes \cE_1}\big]=(I_1\otimes VD_{\hat{T}_{1n}})\tilde{T}_1^*h
 \]
 and
 \[
 CVD_{\hat{T}_{1n}}h=\dtn h\oplus 0_{\cE_1}.
 \]
 The above two together imply,
 \begingroup
	\allowdisplaybreaks
 \begin{align*}
 &AVD_{\hat{T}_{1n}}h+B\big[(I_{1n}\otimes \dtn )(I_1\otimes \tilde{T_n}^*)\tilde{T}_1^*h \oplus 0_{E_{1n}\otimes \cE_1}\big]\\
 =&AVD_{\hat{T}_{1n}}h+B\big[\big(I_{1n}\otimes (\dtn \oplus 0_{\cE_1})\big)(I_1\otimes \tilde{T_n}^*)\tilde{T}_1^*h\big]\\
 =&AVD_{\hat{T}_{1n}}h+B\big[\big(I_{1n}\otimes CVD_{\hat{T}_{1n}}\big)(I_1\otimes \tilde{T_n}^*)\tilde{T}_1^*h\big]\\
 =&(I_1\otimes VD_{\hat{T}_{1n}})\tilde{T}_1^*h
 \end{align*}
 \endgroup
 Define \begin{equation} \label{tauU1} \tau_{U_1^*}:=I_{\cF(E)}\otimes \big(A^*+[I_{1n}\otimes C^*] B^*\big).\end{equation}
More explicitely,
$$ \tau_{U_1^*}=I_{\cF(E)}\otimes \big(P_1^*(I_1\otimes U^*)+[I_{1n}\otimes i_2] (J_2')^*(I_1\otimes U^*)\big).$$
  This map will play an important role.
 
 Note that
 \be\label{tau1map}\tau_{U_1^*}:\cF(E)\otimes E_1\otimes \cD \rightarrow \cF(E)\otimes [\cD \oplus (E_{1n}\otimes \cD)]\subseteq \cF(E)\otimes \cD \ee
 
 and we view here $\cF(E)\otimes E_{1n}\otimes \cD$ as a subspace of $\cF(E)\otimes \cD$.
 
\begin{proposition}\label{P1}
     For the unitary $U_1$, $\tau_{U_1^*}$ is a well defined  isometry.
\end{proposition}
\begin{proof}
	To show that $\tau_{U_1^*}$ is well defined we need to show that $B(I_{1n}\otimes C)$ is well defined. For this, recall that $C=i_2^*$ and, thus, $I_{1n}\otimes C=I_{1n}\otimes i_2^*:(E_{1n}\otimes\Dn )\oplus (E_{1n}\otimes E_1 \otimes \Done ) \oplus (E_{1n}\otimes \cE_1) \rightarrow (E_{1n}\otimes\Dn ) \oplus (E_{1n}\otimes \cE_1) .$
	For $B$ we have $B=(I_1\otimes U)j_2'$ which is defined on  $(E_{1n}\otimes \cD_{\tilde{T}_1})\oplus (E_{1n}\otimes \cE_1)$. Thus $B(I_{1n}\otimes C)$ is well defined.
	The following computation shows that $\tau_{U_1^*}$ is indeed an isometry.
\begin{align*}
    &\tau_{U_1^*}^*\tau_{U_1^*}\\
    =&\Big[I_{\cF(E)}\otimes \big(A+B[I_{1n}\otimes C]\big)\Big]\Big[I_{\cF(E)}\otimes \big(A^*+[I_{1n}\otimes C^*] B^*\big)\Big]\\
    =&\big(I_{\cF(E)}\otimes A\big)\big(I_{\cF(E)}\otimes A^*\big)+\big(I_{\cF(E)}\otimes A\big)\big(I_{\cF(E)}\otimes [I_{1n}\otimes C^*] B^*\big)\\
    &+ \big(I_{\cF(E)}\otimes B[I_{1n}\otimes C]\big)\big(I_{\cF(E)}\otimes A^*\big)+ \big(I_{\cF(E)}\otimes B[I_{1n}\otimes C]\big)\big(I_{\cF(E)}\otimes [I_{1n}\otimes C^*] B^*\big)
\end{align*}
First term and last term of the above sum are equal to $I_{\cF(E)}\otimes AA^*$ and $I_{\cF(E)}\otimes B[I_{1n}\otimes CC^*]B^*$, respectively. Recall now that
\begin{align*}
	B^*:& \,E_1\otimes \cD \to E_{1n}\otimes \cD^{'}\\
	C^*:& \,\cD^{'}\to \cD\\
	A:& \,\cD \to E_1\otimes \cD. 
\end{align*}
It follows that the range of $I_{\cF(E)}\otimes [I_{1n}\otimes C^*] B^*$ is contained in $\cF(E)\otimes E_{1n}\otimes \cD (\subseteq \cF(E)\otimes \cD)$. On this range $I_{\cF(E)}\otimes A$ equals $I_{\cF(E)}\otimes I_{1n}\otimes A$ and, thus, the third term is $I_{\cF(E)}\otimes [I_{1n}\otimes AC^*]B^*$ and the forth term is its adjoint.

Hence, $$\big(I_{\cF(E)}\otimes A\big)\big(I_{\cF(E)}\otimes [I_{1n}\otimes C^*] B^*\big)=I_{\cF(E)}\otimes [I_{1n}\otimes AC^*] B^*.$$
Similarly,
 \[
\big(I_{\cF(E)}\otimes B[I_{1n}\otimes C]\big)\big(I_{\cF(E)}\otimes A^*\big)=I_{\cF(E)}\otimes B[I_{1n}\otimes CA^*].
\]
Therefore,
 \[
\tau_{U_1^*}^*\tau_{U_1^*}=I_{\cF(E_{1n})}\otimes \big(AA^*+[I_{1n}\otimes AC^*] B^*+B[I_{1n}\otimes CA^*]+B[I_{1n}\otimes CC^*]B^*\big).
 \]
 The fact that $U_1$ is a unitary implies 
\[
AC^*=0=CA^*,\quad CC^*=I_{\cD'}\quad\text{and}\quad AA^*+BB^*=I_{E_1\oplus \cD}
\]
and hence,
\[
\tau_{U_1^*}^*\tau_{U_1^*}=I_{\cF(E)\otimes E_1\otimes \cD}.
\]
\end{proof}

We have the following dilation result.
 \begin{theorem}\label{dilation1}
 	$(I_1\otimes\Pi_{1n,V})\tilde{T}_1^*=\tau_{U_1^*}^*\Pi_{1n,V}$.
 \end{theorem}
\begin{proof} 

For $h\in H$, $\eta^{(\beta)}\in E^{\beta}$ and $k \in E_1\otimes \big(D_{\tilde{T}_{1}}\oplus (E_1\otimes D_{\tilde{T}_{n}})\oplus\cE_1 \big)$
 \begingroup
	\allowdisplaybreaks
	\begin{align*}
	&\langle \tau_{U_1^*}^*\Pi_{1n,V}h,\eta^{(\beta)}\otimes k\rangle\\
	=&\langle \Pi_{1n,V}h,\tau_{U_1^*}(\eta^{(\beta)}\otimes k)\rangle\\
	=& \Big\langle \bigoplus_{\alpha \in \bZ_+^{n-1}}\big(I_{E^{\alpha}}\otimes VD_{\hat{T}_{1n}}\big)\tilde{T}_{1n}^{*(\alpha)}h ,
  \quad \big[ I_{\cF(E)}\otimes A^*+\big(I_{\cF(E)}\otimes [I_{1n}\otimes C^*]\big)(I_{\cF(E)}\otimes B^*)\big]\big(\eta^{(\beta)}\otimes k\big) \Big\rangle \\
 =&\Big\langle \bigoplus_{\alpha \in \bZ_+^{n-1}} \big( I_{E^{\alpha}} \otimes VD_{\hat{T}_{1n}}\big)\tilde{T}_{1n}^{*(\alpha)}h,  \Big(\bigoplus_{\alpha \in \bZ_+^{n-1}} \big( I_{E^{\alpha}}\otimes A^*\big)\Big)(\eta^{(\beta)}\otimes k) \Big \rangle\\
 &\hspace{0.5cm} + \Big\langle \bigoplus_{\alpha \in \bZ_+^{n-1}} \big( I_{E^{\alpha}} \otimes VD_{\hat{T}_{1n}}\big)\tilde{T}_{1n}^{*(\alpha)}h,  \Big(\bigoplus_{\alpha \in \bZ_+^{n-1}} 
 \big( I_{E^{\alpha}} \otimes[I_{1n}\otimes C^*]\big)\Big)\Big(\bigoplus_{\alpha \in \bZ_+^{n-1}} 
 \big( I_{E^{\alpha}} \otimes B^*\big)\Big)(\eta^{(\beta)}\otimes k) \Big \rangle\\
  =& \Big\langle\big( I_{E^{\beta}} \otimes VD_{\hat{T}_{1n}}\big)\tilde{T}_{1n}^{*(\beta)}h, \big( I_{E^{\beta}} \otimes A^*\big) (\eta^{(\beta)}\otimes k) \Big \rangle    \\
& \hspace{.5cm} + \Big\langle \bigoplus_{ \alpha= (0,\alpha_2,\ldots,\alpha_{n-1})  \in \bZ_+^{n-1}} \big(I_{E^{\alpha}} \otimes VD_{\hat{T}_{1n}} \big) \tilde{T}_{1n}^{*(\alpha)}h  \quad + \quad  
I_{1n} \otimes \bigoplus_{\alpha \in \bZ_+^{n-1}} \big(I_{E^{\alpha}}\otimes VD_{\hat{T}_{1n}}\big)\tilde{T}_{1n}^{*(\alpha + e_1)}h,\\
&\hspace{3cm}\Big(I_{1n}\otimes\bigoplus_{\alpha \in \bZ_+^{n-1}} \big( I_{E^{\alpha}} \otimes C^*\big)\Big) \Big(\bigoplus_{\alpha \in \bZ_+^{n-1}}  
\big( I_{E^{\alpha}}\otimes B^*\big)\Big) 
(\eta^{(\beta)}\otimes k) \Big \rangle\\
=& \langle\big( I_{E^{\beta}} \otimes AVD_{\hat{T}_{1n}}\big) \tilde{T}_{1n}^{*(\beta)}h, \eta^{(\beta)}\otimes k\big \rangle\\
&\hspace{1cm} + \Big\langle 
I_{1n} \otimes \bigoplus_{\alpha \in \bZ_+^{n-1}} \big(I_{E^{(\alpha)}} \otimes CVD_{\hat{T}_{1n}}\big) \tilde{T}_{1n}^{*(\alpha + e_1)}h, \Big(\bigoplus_{\alpha \in \bZ_+^{n-1}}\big( I_{E^{\alpha}} \otimes B^*\big)\Big)(\eta^{(\beta)} \otimes k) \Big \rangle\\
=& \big\langle\big( I_{E^{\beta}} \otimes AVD_{\hat{T}_{1n}}\big) \tilde{T}_{1n}^{*(\beta)}h, \eta^{(\beta)}\otimes k\big \rangle\\
&\hspace{1cm} + \big\langle I_{1n}\otimes\big(I_{E^{\beta}} \otimes CVD_{\hat{T}_{1n}}\big) \tilde{T}_{1n}^{*(\beta + e_1)}h, \big( I_{E^{\beta}} \otimes B^* \big)(\eta^{(\beta)}\otimes k) \big \rangle \\
=& \big\langle\big( I_{E^{\beta}} \otimes AVD_{\hat{T}_{1n}}\big)\tilde{T}_{1n}^{*(\beta)}h, \eta^{(\beta)}\otimes k \big \rangle \\
&\hspace{1cm} + \big\langle I_{E^{\beta}} \otimes \big[\big(I_{1n} \otimes CVD_{\hat{T}_{1n}}\big) \tilde{T}_{1n}\big]\tilde{T}_{1n}^{*(\beta)}h, \big( I_{E^{\beta}} \otimes B^*\big)(\eta^{(\beta)}\otimes k) \big\rangle \\
=& \langle(I_{E^{\beta}} \otimes AVD_{\hat{T}_{1n}}) \tilde{T}_{1n}^{*(\beta)}h + [I_{E^{\beta}} \otimes B(I_{1n} \otimes CVD_{\hat{T}_{1n}}) \tilde{T}_{1n}^*] \tilde{T}_{1n}^{*(\beta)}h, \eta^{(\beta)}\otimes k \rangle \\
=& \langle\big [I_{E^{\beta}} \otimes (I_1 \otimes VD_{\hat{T}_{1n}}) \tilde{T}_{1}^* \big] \tilde{T}_{1n}^{*(\beta)}h, \eta^{(\beta)}\otimes k \rangle \\
=& \langle\big [I_1\otimes (I_{E^{\beta}} \otimes VD_{\hat{T}_{1n}}) \tilde{T}_{1n}^{*(\beta)}] \tilde{T}_{1}^*h,\eta^{(\beta)} \otimes k \big\rangle\\
=&\langle(I_1\otimes\Pi_{1n,V})\tilde{T}_1^*h,\eta^{(\beta)} \otimes k \rangle. \\
\end{align*}
\endgroup
\end{proof}	

We know that the isometry 
$$U: \Done \oplus (E_n\otimes \Dn )\oplus \cE_2\rightarrow \Dn \oplus (E_1\otimes \Done )\oplus \cE_1   $$ satisfies 

\[ 
U(\dtone h \oplus (I_n \otimes \dtn ) \tilde{T_n}^*h\oplus 0_{\cE_2}) = (\dtn h \oplus (I_1 \otimes \dtone ) \tilde{T_1}^*h\oplus 0_{\cE_1}).  
\]
So the adjoint operator 
$$U^*: \Dn \oplus (E_1\otimes \Done )\oplus \cE_1\rightarrow \Done \oplus (E_n\otimes \Dn )\oplus \cE_2   $$ satisfies

\[ U^*(\dtn h \oplus (I_1 \otimes \dtone ) \tilde{T_1}^*h \oplus 0_{\cE_1}) = (\dtone h \oplus (I_n \otimes \dtn ) \tilde{T_n}^*h \oplus 0_{\cE_2}).  \]
We now need to consider the following operators:
\begin{enumerate} 
\item[(1)] $P_2$ is the projection $$P_2:\Done  \oplus (E_n\otimes \Dn )\oplus \cE_2 \rightarrow 	(E_n\otimes \Dn )\oplus \cE_2$$
\item[(2)] $i_1'$ is the inclusion $$i_1':(E_n \otimes E_1 \otimes \Done ) \rightarrow (E_n\otimes \Dn )\oplus (E_n\otimes E_1 \otimes \Done ) \oplus (E_n \otimes \cE_1) .$$
\item[(3)] $i_1$ is  another inclusion 
$$i_1:\Done \rightarrow \Done \oplus (E_n \otimes \Dn ) \oplus \cE_2$$ so that $i_1^*$ is the projection $$i_1^*:\Done \oplus (E_n \otimes \Dn ) \oplus \cE_2 \rightarrow \Done .$$
	
\end{enumerate}

Recall the map 
$V:\cD_{\hat{T}_{1n}}\rightarrow \Dn \oplus (E_1\otimes \Done )\oplus \cE_1 $ such that 
$$V(D_{\hat{T}_{1n}}h)=\dtn h \oplus (I_1\otimes \dtone )\tilde{T_1}^*h\oplus 0_{\cE_1}.$$

Let $\cD_{1n}=\Dn \oplus (E_1 \otimes \Done )$. With this notation let,

$$U_n:= \left(\begin{array}{cc} \big(I_{E_n\otimes\cD_{1n}}\oplus u_2\big)P_2 U^* & i_1' \\ i_1^*U^* & 0 \end{array}\right) : \left(\begin{array}{c} \Dn  \oplus (E_1\otimes \Done )\oplus \cE_1 \\ (E_n\otimes E_1 \otimes \Done )\end{array}\right) $$  $$\rightarrow \left(\begin{array}{c}  (E_n\otimes \Dn )\oplus (E_n\otimes E_1 \otimes \Done )\oplus ( E_n\otimes \cE_1 )\\ \Done \end{array}\right). $$
So that
$$U_n= \left(\begin{array}{cc} \big(I_{E_n\otimes\cD_{1n}}\oplus u_2\big)P_2 U^* & i_1' \\ i_1^*U^* & 0 \end{array}\right) : \left(\begin{array}{c} \cD \\ E_{1n}\otimes\Done  \end{array}\right) \rightarrow \left(\begin{array}{c}  E_n\otimes \cD\\\Done  \end{array}\right). $$
A straightforward computation shows that $U_n$ is unitary.

\begin{lemma}\label{U2intertwines}
	The operator $U_n$ intertwines the representations of $M$, that is
	$$U_n(\rho_{\cD}(a) \oplus (\varphi_{1n}(a)\otimes I_{\Done}))=((\varphi_{2}(a)\otimes I_{\cD})\oplus \sigma|_{\Done} )U_n$$ for every $a\in M$.
\end{lemma} 
\begin{proof}
	We need to prove the following, for every $a\in M$,
	\begin{enumerate}
		\item[(i)]	$(I_{E_2\otimes \cD_{1n}}\oplus u_2)P_2U^*\rho_{\cD}(a)=(\varphi_2(a)\otimes I_{\cD})(I_{E_2\otimes \cD_{1n}}\oplus u_2)P_2U^*$.
		\item[(ii)] $i_1'(\varphi_{1n}(a)\otimes I_{\Done})=(\varphi_2(a)\otimes I_{\cD})i_1'$.
		
	\end{enumerate} 
	We start with (i) and compute, using  the intertwining property of $U^*$      \\ (see Proposition~\ref{unitary}(4)),
	$$(I_{E_2\otimes \cD_{1n}}\oplus u_2)P_2U^*\rho_{\cD}(a)=(I_{E_2\otimes \cD_{1n}}\oplus u_2)P_2(\sigma(a)\oplus (\varphi_2(a)\otimes I_{\Dn})\oplus \rho_2(a))U^*.$$ Using the definition of $P_2$ this is equal to
	$$(I_{E_2\otimes \cD_{1n}}\oplus u_2)((\varphi_2(a)\otimes I_{\Dn})\oplus \rho_2(a))P_2U^*=((\varphi_2(a)\otimes I_{\Dn})\oplus(\varphi_2(a)\otimes I_{\cE_1})u_2)P_2U^*$$ 
	$$=((\varphi_2(a)\otimes I_{\Dn})\oplus(\varphi_2(a)\otimes I_{\cE_1}))(I_{E_2\otimes \cD_{1n}}\oplus u_2)P_2U^*=(\varphi_2(a)\otimes I_{\cD})(I_{E_2\otimes \cD_{1n}}\oplus u_2)P_2U^*,$$ proving (i).
	
	To prove (ii) we compute
	$$i_1'(\varphi_{1n}(a)\otimes I_{\Done})=((\varphi_{1n}(a)\otimes I_{\Done})\oplus (\varphi_{1n}(a)\otimes I_{\Dn})\oplus (\varphi_2(a)\otimes I_{\cE_1}))i_1'$$  $$=\varphi_2(a) \otimes (I_{E_1\otimes \Done} \oplus I_{\Dn} \oplus I_{\cE_1})i_1'=(\varphi_2(a)\otimes I_{\cD})i_1',$$ proving (ii).
	
	To prove (iii) we again use the intertwining property of $U^*$  as follows.
	$$i_1^*U^*\rho_{\cD}(a)=i_1^*(\sigma(a)\oplus (\varphi_2(a) \otimes I_{\Dn})\oplus \rho_2(a))U^*=\sigma(a)i_1^*U^*.$$
	
\end{proof}

Now, we compute  
\begin{align*}
& U_n\big(VD_{\hat{T}_{1n}}h, (I_{1n}\otimes \dtone )(I_n\otimes \tilde{T_1}^*)\tilde{T_n}^*h\big) \\ 
&=  U_n\big(\dtn h\oplus (I_1\otimes \dtone )\tilde{T}_1^*h\oplus 0_{\cE_1},  (I_{1n}\otimes \dtone )(I_n\otimes \tilde{T}_1^*)\tilde{T_n}^*h \big) \\
&= \left(\begin{array}{cc} \big(I_{E_n\otimes\cD_{1n}}\oplus u_2\big)P_2 U^* & i_1' \\ i_1^*U^* & 0 \end{array}\right) \left(\begin{array}{cc} \dtn h\oplus (I_1\otimes \dtone )\tilde{T}_1^*h\oplus 0_{\cE_1} \\ (I_{1n}\otimes \dtone )(I_n\otimes \tilde{T}_1^*)\tilde{T_n}^*h \end{array}\right)  \\
&= \left(\begin{array}{cc} \big(I_{E_n\otimes\cD_{1n}}\oplus u_2\big)P_2 U^*\big( \dtn h\oplus (I_1\otimes \dtone )\tilde{T}_1^*h \oplus 0_{\cE_1}\big) +  
i_1' \big( (I_{1n}\otimes \dtone )(I_n\otimes \tilde{T}_1^*)\tilde{T_n}^*h \big) \\ 
i_1^*U^*\big( \dtn h\oplus (I_1\otimes \dtone )\tilde{T}_1^*h\oplus 0_{\cE_1} \big) \end{array}\right) \\
&=  \left(\begin{array}{cc} \big(I_{E_n\otimes\cD_{1n}}\oplus u_2\big)P_2 \big( \dtone h\oplus (I_n\otimes \dtn )\tilde{T_n}^*h \oplus 0_{\cE_2}\big) +  
 \big( (I_{1n}\otimes \dtone )(I_n\otimes \tilde{T}_1^*)\tilde{T_n}^*h \oplus 0_{E_n\otimes (\cD_{\tilde{T}_1}\oplus \cE_1)}\big) \\ 
i_1^* \big( \dtone h\oplus (I_n \otimes \dtn )\tilde{T_n}^*h \oplus 0_{\cE_1}\big) \end{array}\right) \\
&= \left(\begin{array}{cc}  0_{E_{1n}\otimes \cD_{\tilde{T_n}}}\oplus (I_n\otimes \dtn )\tilde{T_n}^*h  \oplus 0_{E_n\otimes\cE_1}+  
(I_{1n}\otimes \dtone )(I_n\otimes \tilde{T}_1^*)\tilde{T_n}^*h \oplus 0_{E_n\otimes \cD_{\tilde{T}_1}}\oplus 0_{E_n\otimes\cE_1}\\ 
\dtone  h \end{array}\right) \\
&= \left(\begin{array}{cc}  (I_n\otimes \dtn )\tilde{T_n}^*h  \oplus  
(I_{1n}\otimes \dtone )(I_n\otimes \tilde{T}_1^*)\tilde{T_n}^*h \oplus 0_{E_n\otimes\cE_1}\\ 
\dtone  h \end{array}\right) \\
&= \left( \begin{array}{cc}
     (I_n \otimes VD_{{\tilde T}_{12}}) \tilde{T_n}^* h  \\
     \dtone  h
\end{array} \right).
\end{align*}

That is, $U_n=\begin{bmatrix}A & B\\
C & 0\end{bmatrix}$ has the following property similar as $U_1$,
\begin{equation}\label{ABn}
 AVD_{\hat{T}_{1n}}h+ B[(I_{1n}\otimes \dtone )(I_n\otimes \tilde{T_1}^*)\tilde{T_n}^*h] = (I_n \otimes VD_{{\tilde T}_{1n}}) \tilde{T_n}^* h
\end{equation}
and
\begin{equation}\label{Cn}
CVD_{\hat{T}_{1n}}h=\dtone h.
\end{equation}
Recall that \begin{equation}\label{tauU2}
\tau_{U_n^*}=I_{\cF(E)} \otimes (A^*+[I_{1n}\otimes C^*]B^*)\end{equation} 
where $A^*=UP_2^*\big(I_{E_n\otimes\cD_{1n}}\oplus u_2^*\big)$, $B^*=(i_1')^*$ and $C^*=Ui_1$ so that we can write
$$\tau_{U_n^*}=I_{\cF(E)} \otimes \Big(UP_2^*\big(I_{E_n\otimes\cD_{1n}}\oplus u_2^*\big)+[I_{1n}\otimes Ui_1](i_1')^*\Big).$$ 

The map $I_{1n}\otimes Ui_1$ is defined on $E_{1n}\otimes\Done $. Since this is the range of $(i_1')^*$, $\tau_{U_n^*}$ is well defined.
In fact
\be \label{tau2map} \tau_{U_n^*}:\cF(E)\otimes E_n \otimes \cD \rightarrow \cF(E) \otimes [\cD \oplus (E_{1n}\otimes \cD)]\subseteq \cF(E)\otimes \cD \ee where we view $\cF(E)\otimes E_{1n}\otimes \cD$ as a subspace of $\cF(E)\otimes \cD$.
 
So, as in Proposition \ref{P1} and Theorem \ref{dilation1}, we have the following result,
\begin{theorem}\label{dilation2}
 	$\tau_{U_n^*}$ is an isometry and $(I_n\otimes\Pi_{1n,V})\tilde{T_n}^*=\tau_{U_n^*}^*\Pi_{1n,V}$.
\end{theorem}
\begin{proof} 
	The fact that $\tau_{U_n^*}$ is an isometry follows from the proof of Proposition \ref{P1}.
	For $h\in H$, $\eta^{(\beta)}\in E^{\beta}$ and $k \in E_n\otimes \big(D_{\tilde{T}_{1}}\oplus (E_1\otimes D_{\tilde{T}_{n}}) \big)$
	\begingroup
	\allowdisplaybreaks
	\begin{align*}
	&\langle \tau_{U_n^*}^*\Pi_{1n,V}h,\eta^{(\beta)}\otimes k\rangle\\
	=&\langle \Pi_{1n,V}h,\tau_{U_n^*}(\eta^{(\beta)}\otimes k)\rangle\\
	=& \Big\langle \bigoplus_{\alpha \in \bZ_+^{n-1}}\big(I_{E^{\alpha}}\otimes VD_{\hat{T}_{1n}}\big)\tilde{T}_{1n}^{*(\alpha)}h ,\quad \big[ I_{\cF(E)}\otimes A^*+\big(I_{\cF(E)}\otimes [I_{1n}\otimes C^*]\big)(I_{\cF(E)}\otimes B^*)\big]\big(\eta^{(\beta)}\otimes k\big) \Big\rangle
	\end{align*}
As in the proof of Proposition \ref{dilation1}, the above expression is equal to  
\[ 
\langle(I_{E^{\beta}} \otimes AVD_{\hat{T}_{1n}}) \tilde{T}_{1n}^{*(\beta)}h + [I_{E^{\beta}} \otimes B(I_{1n} \otimes CVD_{\hat{T}_{1n}}) \tilde{T}_{1n}^*] \tilde{T}_{1n}^{*(\beta)}h, \eta^{(\beta)}\otimes k \rangle.
\]
Using the properties of $U_n$, listed in (\ref{ABn}) and (\ref{Cn}), this again equals to
\begin{align*}
	& \langle\big [I_{E^{\beta}} \otimes (I_n \otimes VD_{\hat{T}_{1n}}) \tilde{T}_{n}^* \big] \tilde{T}_{1n}^{*(\beta)}h, \eta^{(\beta)}\otimes k \rangle \\
	=& \langle\big [I_n\otimes (I_{E^{\beta}} \otimes VD_{\hat{T}_{1n}}) \tilde{T}_{1n}^{*(\beta)}] \tilde{T}_{n}^*h,\eta^{(\beta)} \otimes k \big\rangle\\
	=&\langle(I_n\otimes\Pi_{1n,V})\tilde{T}_n^*h,\eta^{(\beta)} \otimes k \rangle. \\
\end{align*}
	\endgroup
\end{proof}	

We now turn to show that $\tau_{U_i^*}$ gives rise to a representation of $E_i$ on $K:=\cF(E)\otimes \cD$ where $\cF(E)$ is as in (\ref{FE}).
It will be convenient to note that 
\begin{equation}\label{K}
	K=\cD\oplus E_{1n}\otimes K \oplus(\bigoplus_{i=2}^{n-1}E_i \otimes K)
	\end{equation} 

	In fact, there is a natural representation $\rho$ on $\cB(K)$, $$\rho: M\to\cB(K)=\cB\Big(\cD \oplus E_{1n}\otimes K \oplus\Big(\bigoplus_{i=2}^{n-1}E_{i}\otimes K\Big)\Big),$$ which is defined by
\begin{align*}
	\rho(a)=\varphi_{\cF(E)}(a)\otimes I_{\cD}=\rho_{\cD}\oplus (\varphi_{1n}(a)\otimes I_{K})\oplus \Big(\bigoplus_{i=2}^{n-1}\varphi_i(a)\otimes I_{K}\Big)
\end{align*}
and we shall show that there are representations $(\rho,M_i)$ of $E_i$ ($i=1,n$) such that $\tau_{U_1^*}=\tilde{M}_1$ and $\tau_{U_n^*}=\tilde{M}_n$ (in the notation of (\ref{covariance})).
\begin{proposition}\label{repn} Let $K$ be as above. There are representations $M_1$ (of $E_1$) and $M_n$ (of $E_2$) on $\cF(E)\otimes\cD$ such that $\tau_{U_1^*}=\tilde{M_1}$ and $\tau_{U_n^*}=\tilde{M_n}$.
\end{proposition}
\begin{proof} 
Let $a\in M$. First, we shall verify,
\begin{equation}\label{commutation1}
	\tilde{L}_1\big(\varphi_{1n}(a)\otimes I_{K}\big) =\rho(a)\tilde{L}_1\quad\&\,\,\, \tilde{L}_i\big(\varphi_{i}(a)\otimes I_{K}\big) =\rho(a)\tilde{L}_i\quad(\forall\,\,i=2,\ldots,n-1)
\end{equation}
Then we shall show
\begin{equation}\label{commutation2}
	\tau_{U_1^*}\big(\varphi_1(a)\otimes I_{K}\big) =\rho(a)\tau_{U_1^*}
\end{equation}
and
\begin{equation}\label{commutation3}
	\tau_{U_n^*}\big(\varphi_n(a)\otimes I_{K}\big) =\rho(a)\tau_{U_n^*}.
\end{equation}
Once we verify these, the proposition will follow.
To verify (\ref{commutation1}) we write, for $\xi\in E_{1n}$, $\xi_i\in E_i$ and $k\in K$,
$$\rho(a)\tilde{L}_1(\xi \otimes k)=\rho(a)(\xi \otimes K)=\varphi_{1n}(a)\xi \otimes k=\tilde{L}_1(\varphi_{1n}(a)\otimes I_{K})(\xi \otimes k) $$
and 
$$\rho(a)\tilde{L}_i(\xi_i \otimes K)=\rho(a)(\xi_i \otimes k)=\varphi_i(a)\xi_i \otimes k=\tilde{L}_i(\varphi_i(a)\otimes I_{K})(\xi_i \otimes k).$$
Hence (\ref{commutation1}) follows.
To prove (\ref{commutation2}), we will use (i) - (ii) that were proved in the proof of Lemma~\ref{Uintertwines}. In fact, we shall use the adjoints, namely,
\begin{enumerate}
	\item [(i*)] $P_1^*(I_1\otimes U^*)(\varphi_1(a)\otimes I_{\cD})=\rho_{\cD}(a)P_1^*(I_1\otimes U^*)$, and
	\item [(ii*)] $(j'_2)^*(I_1\otimes U^*)(\varphi_1(a)\otimes I_{\cD})=(\varphi_{1n}(a)\otimes I_{\cD'})(j_2')^*(I_1\otimes U^*)$.
\end{enumerate}
We now turn to prove (\ref{commutation2}). Note that the right-hand side is defined on $\cF(E)\otimes E_1 \otimes \cD$ but this is equal to $E_1\otimes \cF(E)\otimes \cD$ and so the left-hand side is also well defined. In fact, we can write 
$$E_1\otimes \cF(E)\otimes \cD=E_1\otimes \cD \oplus \Big(\bigoplus_{\alpha\in\bZ_+^{n-1},\,\alpha>0}\big(E_1\otimes E^{\otimes \alpha}\otimes \cD\big)\Big)$$ 
and proceed to prove the equality on $E_1\otimes \cD$ and on $E_1\otimes  E^{\otimes\alpha}\otimes \cD$.  We start with $E_1\otimes \cD$ and fix $\xi\otimes d\in E_1\otimes \cD$.
$$\tau_{U_1^*}(\varphi_1(a)\otimes I_{\cF(E_{12})\otimes \cD})(\xi\otimes d)=\tau_{U_1^*}(\varphi(a)\xi \otimes d)=(P_1^*(I_1\otimes U^*) + (I_{1n}\otimes i_2)(j_2')^*(I_1\otimes U^*))(\varphi(a)\xi \otimes d).$$  Using (i*) and (ii*), this is equal to
\begin{align*}
	\big(\rho_{\cD}(a)&P_1^*(I_1\otimes U^*)+ (I_{1n}\otimes i_2)(\varphi_{1n}(a)\otimes I_{\cD'})(j_2')^*(I_1\otimes U^*)\big)(\xi\otimes d)\\
	=&\big(\rho_{\cD}(a)P_1^*(I_1\otimes U^*)+(\varphi_{1n}(a)\otimes I_{\cD})
	(j_2')^*(I_1\otimes U^*)\big)(\xi\otimes d)\\
	=&\big((\rho_{\cD}(a)\oplus(\varphi_{1n}(a)\otimes I_{\cD})\big)\big(P_1^*(I_1\otimes U^*) + (I_{1n}\otimes i_2)(j_2')^*(I_1\otimes U^*)\big)(\xi \otimes d)\\
	=&\rho(a)\tau_{U_1^*}(\xi \otimes d)
\end{align*}
proving the statement (\ref{commutation2}) applied to $E_1\otimes \cD$.
Now, we apply the left hand side of (\ref{commutation2}) to $\xi' \otimes d\in E_1\otimes E^{\otimes \alpha} \otimes \cD$ (for $\alpha\in\bZ_+^{n-1},\alpha>0$). But $\xi'$ can also be written as $\xi'\in E^{\otimes\alpha}\otimes E_1$ so we apply it to $\eta \otimes \xi$ where $\eta\in E^{\otimes \alpha} $ and $\xi\in E_1$ to get
$$\tau_{U_1^*}\big(\varphi_1(a)\otimes I_{E^{\otimes \alpha} \otimes \cD}\big)(\xi'\otimes d)=\tau_{U_1^*}\big(\varphi_1(a)\otimes I_{E^{\otimes \alpha} \otimes \cD}\big)(\eta \otimes \xi\otimes d)$$  $$=\tau_{U_1^*}\big(\varphi_{E^{\otimes\alpha}}(a)\eta \otimes \xi \otimes d\big)=\varphi_{E^{\otimes\alpha}}(a)\eta \otimes \tau_{U_1^*}(\xi \otimes d)$$ $$=\rho(a)(\eta \otimes \tau_{U_1^*}(\xi \otimes d))=\rho(a)\tau_{U_1^*}(\eta \otimes \xi \otimes d)=\rho(a)\tau_{U_1^*}(\xi' \otimes d).$$
This proves (\ref{commutation2}).

To prove (\ref{commutation3}), we will use (i) - (ii) that were proved in the proof of Lemma~\ref{U2intertwines}. In fact, we shall use the adjoints, namely,
\begin{enumerate}
	\item [(i*)]$ \rho_{\cD}(a)UP_2^*(I_{E_2\otimes \cD_{1n}}\oplus u_2^*)=UP_2^*(I_{E_2\otimes \cD_{1n}}\oplus u_2^*)(\varphi_2(a)\otimes I_{\cD})$, and
	\item [(ii*)] $(i_1')^*(\varphi_2(a)\otimes I_{\cD})=(\varphi_{1n}(a) \otimes I_{\Done})(i_1')^*$ 
\end{enumerate}

As in the proof of (\ref{commutation2}) above, we shall first prove it for $\xi \otimes d\in E_2\otimes \cD$.

So, we compute
$$\tau_{U_n^*}(\varphi_2(a)\xi \otimes d)=UP_2^*(I_{E_2\otimes \cD_{1n}}\oplus u_2^*)(\varphi_2(a)\xi \otimes d) \oplus (I_{1n}\otimes Ui_1)(i_1')^*(\varphi_2(a)\xi\otimes d).$$

for the first summand we use (i*) to find that it is equal to $$\rho_{\cD}(a)UP_2^*(I_{E_2\otimes \cD_{1n}}\oplus u_2^*)(\xi \otimes d)=\rho(a)UP_2^*(I_{E_2\otimes \cD_{1n}}\oplus u_2^*)(\xi \otimes d).$$ For the second summand, we use (ii*) to write it as
$$(\varphi_{1n}(a)\otimes Ui_1)(i_1')^*(\xi\otimes d)$$  and this is equal to
$$\rho(a)(I_{1n}\otimes Ui_1)(i_1')^*(\xi \otimes d).$$
Putting it together, we get (\ref{commutation3}) for $E_2\otimes \cD)$.

Now, we apply the left hand side of (\ref{commutation2}) to $\xi' \otimes d\in E_2\otimes E^{\otimes \alpha} \otimes \cD$ (for $\alpha\in\bZ_+^{n-1},\alpha>0$). But  $E_2\otimes E^{\otimes \alpha}= E^{\otimes \alpha}\otimes E_2$ so it suffices to apply it to $\eta \otimes \xi$ where $\eta\in E^{\otimes \alpha} $ ($\alpha>0$) and $\xi\in E_2$ to get
$$\tau_{U_n^*}(\varphi_2(a)\otimes I_{ E_2\otimes E^{\otimes \alpha} \otimes \cD})(\xi'\otimes d)=\tau_{U_n^*}(\varphi_{E^{\otimes \alpha}}(a)\eta \otimes \xi \otimes d)=\rho(a)\tau_{U_n^*}(\xi'\otimes d)$$ proving (\ref{commutation3}). 
\end{proof}

Now we consider the map $\tau_{U_1^*}(I_1\otimes\tau_{U_n^*})$ which is a product of the maps $\tau_{U_n^*}:\cF(E)\otimes E_n \otimes \cD \rightarrow \cF(E)\otimes \cD$ and $I_1\otimes \tau_{U_n^*}:\cF(E)\otimes E_1 \otimes E_n \otimes \cD \rightarrow \cF(E)\otimes E_1 \otimes \cD$ respectively. Therefore, their product is well-defined and it is 
$$\tau_{U_1^*}(I_1\otimes\tau_{U_n^*}):\cF(E) \otimes E_{1n}\otimes \cD \rightarrow \cF(E)\otimes \cD $$ where $\cF(E)\otimes E_{1n}\otimes \cD$ can be viewed as a subspace of $\cF(E)\otimes \cD$.

For the following theorem we recall the definition of $\tilde{L}_1$ from (\ref{Ltilde}) with the replacement of the coefficient space $H$ by $\cD$ and so the map  
$\tilde{L}_1:E_{1n}\otimes\cF(E)\otimes \cD \to \cF(E)\otimes \cD$  is defined by
\be \label{L1tilde}
\tilde{L}_1(\xi_{1n}\otimes k):=L_1(\xi_{1n})k=\xi_{1n}\otimes k\quad \quad \quad (\xi_1 \in E_{1n}, k\in \cF(E)\otimes \cD).
\ee
It is clear from the definition itself that $\tilde{L}_1$ is the inclusion of $E_{1n}\otimes\cF(E)\otimes \cD$ into 
$ \cF(E)\otimes \cD$. 
Now, we are ready to state the factorization results for the left-creation operator $\tilde{L}_1$. 
\begin{theorem}\label{tau12}
	$$\tau_{U_1^*}(I_1\otimes\tau_{U_n^*})=\tilde{L}_1.	$$
\end{theorem}
\begin{proof}
	To calculate the map $\tau_{U_1^*}(I_1\otimes\tau_{U_n^*})$, we recall that 
	\[
	P_1:\Dn  \oplus (E_1\otimes \Done ) \oplus \cE_1\rightarrow 	E_{1n}\otimes \Dn  \oplus (E_1\otimes \Done )\oplus E_1\otimes\cE_2
	\]
	is the projection onto the space $E_1\otimes \Done $ and 
	\[
	P_2: (E_n\otimes \Dn )\oplus \Done  \oplus \cE_2\rightarrow 	(E_n\otimes \Dn )\oplus E_{1n} \otimes \Done \oplus \cE_2
	\]
	is the projection onto the space $(E_n\otimes \Dn )\oplus \cE_2$. Now, observe that
	\begin{align*}
		I_1\otimes i_1^{'*}=I_{1n} \otimes \bar{i}_1^{'*}, 
	\end{align*}
	where $\bar{i}_1^{'}: (E_1 \otimes \cD_{\tilde{T_n}}) \to  (E_1 \otimes \cD_{\tilde{T_n}}) \oplus \cD_{\tilde{T}_1} \oplus \cE_1 $ is an inclusion. 
	Using this we have 
	\begingroup
	\allowdisplaybreaks
	\begin{align*}
		& \Big[I_{\cF(E)} \otimes P_1^*(I_1 \otimes U^*) \Big] 
		\Big[I_{\cF(E)} \otimes I_1\otimes \big(I_{1n}\otimes Ui_1 \big)i_1^{'*} \Big]   \\
		&= \Big[ \bigoplus_{\alpha \in \bZ_+^{n-1}} I_{E^{\alpha}} \otimes P_1^*(I_1 \otimes U^*) \Big]  \Big[\bigoplus_{\alpha \in \bZ_+^{n-1}} I_{E^{\alpha}} \otimes \big( I_{1n}\otimes I_1\otimes  Ui_1 \big)(I_{1n}\otimes \bar{i}_1^{'*}) \Big]\\
		&= \Big[ \bigoplus_{\alpha \in \bZ_+^{n-1}} I_{E^{\alpha}} \otimes P_1^*(I_1 \otimes U^*) \Big]  \Big[\bigoplus_{\alpha \in \bZ_+^{n-1}} I_{E^{(\alpha + e_1)}} \otimes \big(  I_1\otimes  Ui_1 \big) \bar{i}_1^{'*} \Big]\\
		&= \bigoplus_{\alpha \in \bZ_+^{n-1}} I_{E^{(\alpha + e_1)}} \otimes P_1^*(I_1 \otimes U^*) \big( I_1\otimes  Ui_1 \big)\bar{i}_1^{'*} \\
		&= \bigoplus_{\alpha \in \bZ_+^{n-1}} I_{E^{\alpha}} \otimes I_{1n} \otimes P_1^* \big( I_1\otimes  i_1 \big) \bar{i}_1^{'*} \\
		&= I_{\cF(E)} \otimes \Big( I_{1n} \otimes P_1^* \big( I_1\otimes  i_1 \big) \Big) \Big( I_{1n} \otimes  \bar{i}_1^{'*} \Big)
		\\ 
		&= I_{\cF(E)} \otimes \Big( I_{1n} \otimes P_1^* \big( I_1\otimes  i_1 \big) \Big) \Big( I_1 \otimes i_1^{'*} \Big). 
	\end{align*}
	Similarly, since $U$ is an unitary, we also have 
	\begin{align*}
		&\Big[ I_{\cF(E)} \otimes (I_{1n} \otimes i_2) j_2^{'*} (I_1 \otimes U^*) \Big]  \Big[I_{\cF(E)} \otimes I_1\otimes \big(I_{1n}\otimes Ui_1 \big)i_1^{'*} \Big]   \\
		&= \Big[ \bigoplus_{\alpha \in \bZ_+^{n-1}} I_{E^{\alpha}} \otimes (I_{1n} \otimes i_2) j_2^{'*} (I_1 \otimes U^*) \Big]  
		\Big[\bigoplus_{\alpha \in \bZ_+^{n-1}} I_{E^{\alpha}} \otimes I_{1n}\otimes \big( I_1\otimes  Ui_1 \big)\bar{i}_1^{'*} \Big] \\
		&= \Big[ \bigoplus_{\alpha \in \bZ_+^{n-1}} I_{E^{\alpha}} \otimes (I_{1n} \otimes i_2) j_2^{'*} (I_1 \otimes U^*) \Big]  
		\Big[\bigoplus_{\alpha \in \bZ_+^{n-1}} I_{E^{(\alpha + e_1)}} \otimes \big( I_1\otimes  Ui_1 \big)\bar{i}_1^{'*} \Big]\\
		&= \bigoplus_{\alpha \in \bZ_+^{n-1}} I_{E^{(\alpha + e_1)}} \otimes (I_{1n} \otimes i_2) j_2^{'*} (I_1 \otimes U^*) 
		\big( I_1\otimes  Ui_1 \big)\bar{i}_1^{'*} \\
		&= \bigoplus_{\alpha \in \bZ_+^{n-1}} I_{E^{(\alpha + e_1)}} \otimes (I_{1n} \otimes i_2) j_2^{'*}  ( I_1\otimes  i_1 ) \bar{i}_1^{'*} \\
		&= I_{\cF(E)} \otimes \Big( I_{1n} \otimes (I_{1n} \otimes i_2) j_2^{'*}  ( I_1\otimes  i_1  \big) \Big) \Big( I_{1n} \otimes  \bar{i}_1^{'*} \Big) \\
		&= I_{\cF(E)} \otimes \Big( I_{1n} \otimes (I_{1n} \otimes i_2) j_2^{'*}  ( I_1\otimes  i_1  \big) \Big) \Big( I_1 \otimes i_1^{'*} \Big). 
	\end{align*}
	By using the above identities we compute the product of transfer functions corresponding to the unitaries $U_1$ and $U_n$.
	\begin{align*}
		&\tau_{U_1^*}(I_1\otimes\tau_{U_n^*})\\
		&= \Big[I_{\cF(E)}\otimes P_1^*(I_1\otimes U^*)\Big] \Big[I_{\cF(E)}\otimes I_1\otimes UP_2^* \big(I_{E_n\otimes\cD_{1n}}\oplus u_2^*\big) \Big]\\
		&\hspace{2cm} + \Big[I_{\cF(E)} \otimes P_1^*(I_1 \otimes U^*) \Big] 
		\Big[I_{\cF(E)} \otimes I_1\otimes \big(I_{1n}\otimes Ui_1 \big)i_1^{'*} \Big] \\
		&\hspace{2cm}+ I_{\cF(E)}\otimes \Big[ (I_{1n}\otimes i_2)j_2^{'*}(I_1\otimes U^*)\Big(I_1\otimes UP_2^* \big(I_{E_n\otimes\cD_{1n}}\oplus u_2^*\big)\Big) \Big]  \\
		&\hspace{2cm}+ \Big[ I_{\cF(E)} \otimes (I_{1n} \otimes i_2) j_2^{'*} (I_1 \otimes U^*) \Big]  \Big[I_{\cF(E_{1n})} \otimes I_1\otimes \big(I_{1n}\otimes Ui_1 \big)i_1^{'*} \Big] \\
		&= I_{\cF(E)}\otimes P_1^*\Big(I_1\otimes P_2^*\big(I_{E_n\otimes\cD_{1n}}\oplus u_2^*\big)\Big) + I_{\cF(E)} \otimes \Big( I_{1n} \otimes P_1^* \big( I_1\otimes  i_1 \big) \Big) \Big( I_1 \otimes i_1^{'*} \Big) \\
		&\hspace{2cm}+ I_{\cF(E)}\otimes (I_{1n}\otimes i_2)j_2^{'*}\Big(I_1\otimes P_2^*\big(I_{E_n\otimes\cD_{1n}}\oplus u_2^*\big)\Big) \\
		&\hspace{2cm}+  I_{\cF(E)} \otimes \Big( I_{1n} \otimes (I_{1n} \otimes i_2) j_2^{'*}  ( I_1\otimes  i_1  \big) \Big) \Big( I_1 \otimes i_1^{'*} \Big) \\
		&=I_{\cF(E)}\otimes\Big[P_1^*\Big(I_1\otimes P_2^*\big(I_{E_n\otimes\cD_{1n}}\oplus u_2^*\big)\Big) + \big(I_{1n}\otimes P_1^*(I_1\otimes i_1)\big)(I_1\otimes i_1^{'*})\\& \hspace{2cm} + 
		(I_{1n}\otimes i_2)j_2^{'*} \Big(I_1\otimes P_2^*\big(I_{E_n\otimes\cD_{1n}}\oplus u_2^*\big)\Big)  + 
		\big(I_{1n}\otimes(I_{1n}\otimes i_2)j_2^{'*}(I_1\otimes i_1)\big)(I_1\otimes i_1^{'*}) \Big].
	\end{align*}
	\endgroup
	For the first term, note that, 
	\[
	\text{range}\Big(I_1 \otimes \big(I_{E_n\otimes\cD_{1n}}\oplus u_2^*\big)\Big) = (E_{1n} \otimes \cD_{\tilde{T}_1}) \oplus (E_1\otimes E_{1n} \otimes \cD_{\tilde{T_n}})\oplus (E_1\otimes \cE_2)
	\]
	and
	\begin{align*}
		(I_1 \otimes P_2^*)\big((E_{1n} \otimes \cD_{\tilde{T}_1}) \oplus (E_1\otimes E_{1n} \otimes \cD_{\tilde{T_n}})\oplus (E_1\otimes \cE_2)\big) = &(E_{1n} \otimes \cD_{\tilde{T}_1}) \oplus (E_1\otimes \cE_2)\\\subseteq &(E_1\otimes \Done )\oplus (E_{1n} \otimes \Dn ) \oplus (E_1\otimes \cE_2) 
	\end{align*}
	and $P_1$ is a projection on $E_1 \otimes \Done $. So  $(E_{1n} \otimes \cD_{\tilde{T}_1})\oplus (E_1\otimes \cE_2) \subseteq \text{ker}P_1^*$ and hence $$P_1^*\Big(I_1\otimes P_2^*\big(I_{E_n\otimes\cD_{1n}}\oplus u_2^*\big)\Big) = 0.$$ 
	For the second term, we already have observed that 
	\[
	\big(I_{1n}\otimes P_1^*(I_1\otimes i_1)\big)(I_1\otimes i_1^{'*}) = I_{1n}\otimes \Big(P_1^*(I_1\otimes i_1) \bar{i}_1^{'*} \Big) .
	\]
	Now,
	\[ 
	(E_1 \otimes \cD_{\tilde{T_n}})\oplus \cD_{\tilde{T_n}} \oplus \cE_1 \xrightarrow{\bar{i}_1^{'*}} \ (E_1 \otimes \cD_{\tilde{T_n}}) \xrightarrow{I_1 \otimes i_1}  \big( E_1 \otimes \cD_{\tilde{T_n}} \big) \oplus \big( E_{1n} \otimes \cD_{\tilde{T}_1} \big)  \oplus \big( E_1 \otimes \cE_2 \big). 
	\] 
	This tells us that $(I_1\otimes i_1) \bar{i}_1^{'*}$ is an inclusion map with the range $E_1 \otimes \cD_{\tilde{T_n}}$. Since $P_1$ is a projection with range $E_1\otimes \Done $, so  $\Big(P_1^*(I_1\otimes i_1) \bar{i}_1^{'*} \Big)$ is a projection with range $E_1 \otimes \cD_{\tilde{T_n}}$ and hence 
	\[ 
	I_{1n}\otimes \Big(P_1^*(I_1\otimes i_1)\bar{i}_1^{'*} \Big) = P_{E_{1n} \otimes E_1 \otimes \cD_{\tilde{T_n}}}.
	\]
	Now we will calculate the third term. Observe that 
	\[
	(I_{1n}\otimes i_2)j_2^{'*}\Big(I_1\otimes P_2^*\big(I_{E_n\otimes\cD_{1n}}\oplus u_2^*\big)\Big)=(I_{1n}\otimes i_2)\Big(I_{E_{1n}\otimes\Dn }\oplus (I_1\otimes u_2)\Big)i_2^{'*}(I_1\otimes P_2^*)\Big(I_{E_{1n}\otimes\cD_{1n}}\oplus (I_1\otimes u_2^*)\Big).
	\]
	The range of $I_{E_{1n}\otimes\cD_{1n}}\oplus (I_1\otimes u_2^*)$ is $(E_{1n} \otimes \Dn )\oplus (E_1\otimes E_{1n} \otimes \Done )\oplus (E_1\otimes\cE_2)$. Also,
	\[
	(E_{1n} \otimes \Dn )\oplus (E_1\otimes E_{1n} \otimes \Done )\oplus (E_1\otimes\cE_2)\xrightarrow{I_1\otimes P_2^*} (E_1 \otimes \Done )\oplus (E_{1n} \otimes\Dn )\oplus (E_1\otimes\cE_2) \xrightarrow{i_2^{'*}} (E_{1n} \otimes \Dn )\oplus (E_1\otimes\cE_2).
	\]
	From the above inclusion, it follows that
	\[
	i_2^{'*}(I_1\otimes P_2^*)=P_{(E_{1n} \otimes \Dn )\oplus (E_1\otimes\cE_2)}.
	\]
	Hence, keeping in the mind that $i_2: \Dn \oplus \cE_1\to \cD $ is an inclusion map we conclude
	\begin{align*}
		(I_{1n}\otimes i_2)&j_2^{'*}\Big(I_1\otimes P_2^*\big(I_{E_n\otimes\cD_{1n}}\oplus u_2^*\big)\Big)\\=&(I_{1n}\otimes i_2)\Big(I_{E_{1n}\otimes\Dn }\oplus (I_1\otimes u_2)\Big)P_{(E_{1n} \otimes \Dn )\oplus (E_1\otimes\cE_2)}\Big(I_{E_{1n}\otimes\cD_{1n}}\oplus (I_1\otimes u_2^*)\Big)\\= &P_{(E_{1n} \otimes \Dn )\oplus (E_{1n}\otimes \cE_1)}.
	\end{align*}
	For the last term, observe that $\text{ran}(I_1\otimes i_1)=E_1\otimes \Done $ which lies in the kernel of $j_2^{'*}$. So, $j_2^{'*}(I_1\otimes i_1)=0$ and hence 
	\[
	\big(I_{1n}\otimes(I_{1n}\otimes i_2)j_2^{'*}(I_1\otimes i_1)\big)(I_1\otimes i_1^{'*}).
	\]
	Therefore, 
	\begin{align*}
		\tau_{U_1^*}(I_1\otimes\tau_{U_n^*})=&I_{\cF(E)}\otimes
		\Big[0 + P_{E_{1n} \otimes E_1 \otimes \cD_{\tilde{T_n}}} + P_{(E_{1n} \otimes \Dn )\oplus (E_{1n}\otimes \cE_1)}+0\Big] \\
		=&I_{\cF(E)}\otimes I_{E_{1n} \otimes \big((E_1\otimes \cD_{\tilde{T_n}})\oplus\cD_{\tilde{T}_1}\oplus \cE_1\big)}\\=&\tilde{L}_1.
	\end{align*}

\end{proof} 

Below, we state another kind of factorization result for the left creation operator $\tilde{L}_1$.
\begin{theorem}\label{tau21}
	$$\tau_{U_n^*}(I_n\otimes\tau_{U_1^*})=\tilde{L}_1$$
\end{theorem}

\begin{proof}
	Observe that,
	\begingroup
	\allowdisplaybreaks
	\begin{align*}
		I_n\otimes j_2^{'*} & = I_n \otimes \Big[ \Big(I_{E_{1n} \otimes \cD_{\tilde{T}_1} } \oplus (I_1 \otimes u_2) \Big) \circ i_2^{'*} \Big] \\
		& = \Big( I_n \otimes \Big[ I_{E_{1n} \otimes \cD_{\tilde{T}_1} } \oplus (I_1 \otimes u_2) \Big] \Big) \Big(I_n \otimes i_2^{'*} \Big) \\
		&= \Big[ (I_n \otimes I_{1n} \otimes I_{\cD_{\tilde{T}_1} } ) \oplus (I_{1n} \otimes u_2) \Big] \Big(I_{1n} \otimes \bar{i}_1^{'*} \Big) \\
		&=  \Big[ I_{1n} \otimes \Big( (I_n \otimes  I_{\cD_{\tilde{T}_1} } ) \oplus  u_2 \Big) \Big] \Big(I_{1n} \otimes \bar{i}_2^{'*} \Big) \\
		&= I_{1n} \otimes \Big( I_{E_n \otimes  \cD_{\tilde{T}_1} }  \oplus  u_2 \Big)   \bar{i}_2^{'*}, 
	\end{align*}
	\endgroup
	where $\bar{i}_2^{'}: (E_n \otimes \cD_{\tilde{T}_1}) \oplus \cE_2 \to  (E_n \otimes \cD_{\tilde{T}_1}) \oplus \cD_{\tilde{T_n}} \oplus \cE_2 $ is an inclusion. And so 
	\begingroup
	\allowdisplaybreaks
	\begin{align*}
		I_n\otimes \big(I_{1n}\otimes i_2 \big)j_2^{'*} \big(I_1 \otimes U^*) 
		&= \big(I_{1n} \otimes I_n \otimes i_2 \big) \big( I_n \otimes j_2^{'*} \big) \big(I_{1n} \otimes U^*) \\
		&= \big(I_{1n} \otimes I_n \otimes i_2 \big) \big( I_{1n} \otimes \Big( I_{E_n \otimes  \cD_{\tilde{T}_1} }  \oplus  u_2 \Big)   \bar{i}_2^{'*} \big) \big(I_{1n} \otimes U^* \big) \\
		&= I_{1n} \otimes \big(I_n \otimes i_2 \big)  \big(  I_{E_n \otimes  \cD_{\tilde{T}_1} }  \oplus  u_2 \big)   \bar{i}_2^{'*}U^*  
	\end{align*}
	\endgroup
	
	Using these above identities we have 
	\begingroup
	\allowdisplaybreaks
	\begin{align*}
		& \Big[I_{\cF(E)} \otimes UP_2^*\big(I_{E_n\otimes\cD_{1n}}\oplus u_2^*\big) \Big] 
		\Big[I_{\cF(E)} \otimes I_n \otimes \big(I_{1n}\otimes i_2 \big)j_2^{'*}(I_1 \otimes U^*) \Big]   \\
		&= \Big[ \bigoplus_{\alpha \in \bZ_+^{n-1}} I_{E^{\alpha}} \otimes UP_2^* \big(I_{E_n\otimes\cD_{1n}}\oplus u_2^*\big)\Big]  \Big[\bigoplus_{\alpha \in \bZ_+^{n-1}} I_{E^{(\alpha + e_1)}} \otimes \big( I_n \otimes  i_2 \big)  \big(  I_{E_n \otimes  \cD_{\tilde{T}_1} }  \oplus  u_2 \big)   \bar{i}_2^{'*} U^* \Big] \\
		&= \bigoplus_{\alpha \in \bZ_+^{n-1}} I_{E^{(\alpha + e_1)}} \otimes UP_2^*\big(I_{E_n\otimes\cD_{1n}}\oplus u_2^*\big) \big( I_n \otimes i_2 \big)  \big(  I_{E_n \otimes  \cD_{\tilde{T}_1} }  \oplus  u_2 \big)   \bar{i}_2^{'*}U^* \\
		&= \bigoplus_{\alpha \in \bZ_+^{n-1}} I_{E^{\alpha}} \otimes I_{1n} \otimes UP_2^*\big(I_{E_n\otimes\cD_{1n}}\oplus u_2^*\big) \big( I_n \otimes i_2 \big)  \big(  I_{E_n \otimes  \cD_{\tilde{T}_1} }  \oplus  u_2 \big)   \bar{i}_2^{'*}U^* \\
		&= I_{\cF(E)} \otimes \Big( I_{1n} \otimes UP_2^*\big(I_{E_n\otimes\cD_{1n}}\oplus u_2^*\big) \Big) 
		\Big( I_{1n} \otimes \big( I_n\otimes  i_2 \big) \big(  I_{E_n \otimes  \cD_{\tilde{T}_1} }  \oplus  u_2 \big)   \bar{i}_2^{'*}U^* \Big)
		\\ 
		&= I_{\cF(E)} \otimes \Big( I_{1n} \otimes UP_2^* \big(I_{E_n\otimes\cD_{1n}}\oplus u_2^*\big) \Big) 
		\Big( I_n\otimes \big(I_{1n}\otimes i_2 \big)j_2^{'*} \big(I_1 \otimes U^*)  \Big)\\
		&= I_{\cF(E)} \otimes \Big( I_{1n} \otimes U \Big) \Big( I_{1n} \otimes P_2^*\big(I_{E_n\otimes\cD_{1n}}\oplus u_2^*\big) \Big) \Big( I_n\otimes \big(I_{1n}\otimes i_2 \big)j_2^{'*} \Big) \Big(I_{1n} \otimes U^* \Big). \end{align*}
	\endgroup
	Similarly, we also have 
	\begingroup
	\allowdisplaybreaks
	\begin{align*}
		&\Big[ I_{\cF(E)} \otimes (I_{1n} \otimes Ui_1) i_1^{'*}  \Big]  \Big[I_{\cF(E)} \otimes I_n \otimes \big(I_{1n}\otimes i_2 \big)j_2^{'*} (I_1 \otimes U^*) \Big]   \\
		&= \big( I_{\cF(E)} \otimes I_{1n} \otimes U \big) 
		\Big[ I_{\cF(E)} \otimes (I_{1n} \otimes i_1) i_1^{'*}  \Big]  \Big[I_{\cF(E)} \otimes I_n \otimes \big(I_{1n}\otimes i_2 \big)j_2^{'*}  \Big]  
		\big( I_{\cF(E)} \otimes I_{1n} \otimes U^* \big)   \\
		&= \big( I_{\cF(E)} \otimes I_{1n} \otimes U \big) \Big[ \bigoplus_{\alpha \in \bZ_+^{n-1}} I_{E^{\alpha}} \otimes (I_{1n} \otimes i_1) i_1^{'*}  \Big] \\  
		& \hspace{5cm} \Big[\bigoplus_{\alpha \in \bZ_+^{n-1}} I_{E^{(\alpha + e_1)}} \otimes \big( I_n \otimes  i_2 \big)  
		\big(  I_{E_n \otimes  \cD_{\tilde{T}_1} }  \oplus  u_2 \big)   \bar{i}_2^{'*} \Big]    
		\big( I_{\cF(E)} \otimes I_{1n} \otimes U^* \big) \\
		&=  \big( I_{\cF(E)} \otimes I_{1n} \otimes U \big) 
		\Big[ \bigoplus_{\alpha \in \bZ_+^{n-1}} I_{E^{(\alpha + e_1)}} \otimes (I_{1n} \otimes i_1) i_1^{'*} (I_n \otimes i_2)  \big(  I_{E_n \otimes  \cD_{\tilde{T}_1} }  \oplus  u_2 \big)   \bar{i}_2^{'*} \Big] 
		\big( I_{\cF(E)} \otimes I_{1n} \otimes U^* \big)  \\
		&= \big( I_{\cF(E)} \otimes I_{1n} \otimes U \big) \Big[ \bigoplus_{\alpha \in \bZ_+^{n-1}} I_{E^{(\alpha}} \otimes I_{1n} \otimes (I_{1n} \otimes i_1) i_1^{'*} (I_n \otimes i_2)  
		\big(  I_{E_n \otimes  \cD_{\tilde{T}_1} }  \oplus  u_2 \big)   \bar{i}_2^{'*} \Big]  \big( I_{\cF(E)} \otimes I_{1n} \otimes U^* \big) \\
		&= I_{\cF(E)} \otimes \Big( I_{1n} \otimes I_{1n} \otimes U\Big)\Big( I_{1n} \otimes (I_{1n} \otimes i_1) i_1^{'*}  \Big) \Big( I_n \otimes \big( I_{1n} \otimes  i_2  \big)  j_2^{'*} \Big) \big( I_{1n} \otimes U^* \big) . 
	\end{align*}
	\endgroup
	Combining all the above identities we compute the following
	\begingroup
	\allowdisplaybreaks
	\begin{align*}
		&\tau_{U_n^*}(I_n\otimes\tau_{U_1^*})\\
		&=\Big(I_{\cF(E)}\otimes UP_2^* \big(I_{E_n\otimes\cD_{1n}}\oplus u_2^*\big)\Big)\Big(I_{\cF(E)}\otimes I_n\otimes P_1^*(I_1\otimes U^*)\Big)\\
		&\hspace{2cm} +  \Big(I_{\cF(E)} \otimes UP_2^* \big(I_{E_n\otimes\cD_{1n}}\oplus u_2^*\big)\Big) \Big(I_{\cF(E)}\otimes I_n\otimes (I_{1n}\otimes i_2)j_2^{'*}(I_1\otimes U^*)\Big)\\
		&\hspace{2cm} + \Big(I_{\cF(E)} \otimes [I_{1n}\otimes Ui_1]i_1^{'*}\Big) \Big(I_{\cF(E)}\otimes I_n\otimes P_1^*(I_1\otimes U^*)\Big)\\
		&\hspace{2cm} + \Big(I_{\cF(E)}\otimes [I_{1n}\otimes Ui_1]i_1^{'*}\Big) 
		\Big(I_{\cF(E)}\otimes I_n\otimes (I_{1n}\otimes i_2)j_2^{'*}(I_1\otimes U^*)\Big)\\
		&=I_{\cF(E)}\otimes UP_2^*\big(I_{E_n\otimes\cD_{1n}}\oplus u_2^*\big) \big( I_n\otimes P_1^* \big) (I_{1n}\otimes U^*)\\
		&\hspace{2cm} +  I_{\cF(E)} \otimes \Big( I_{1n} \otimes U \Big) \Big( I_{1n} \otimes P_2^*\big(I_{E_n\otimes\cD_{1n}}\oplus u_2^*\big) \Big) \Big( I_n\otimes \big(I_{1n}\otimes i_2 \big)j_2^{'*} \Big) \Big(I_{1n} \otimes U^* \Big)\\
		&\hspace{2cm}  +   I_{\cF(E)}\otimes \big(I_{1n}\otimes U \big) 
		\big( I_{1n} \otimes i_1 \big) i_1^{'*} \big( I_n\otimes P_1^* \big) \big(I_{1n} \otimes U^* \big)\\
		&\hspace{2cm}  +  I_{\cF(E)} \otimes \big( I_{1n} \otimes I_{1n} \otimes U \big) \Big( I_{1n} \otimes (I_{1n} \otimes i_1) i_1^{'*}  \Big) \Big( I_n \otimes \big( I_{1n} \otimes  i_2  \big)  j_2^{'*} \Big) \big(I_{1n} \otimes U^* \big).
	\end{align*}
	\endgroup
	
	For the first term, note that, 
	\[ 
	\text{range}(I_n \otimes P_1^*) = E_{1n} \otimes \cD_{\tilde{T_n}} \subseteq (E_n\otimes \Dn )\oplus (E_{1n} \otimes \Done ) \oplus (E_n\otimes \cE_1) 
	\] and it follows that $\text{range}\big(I_{E_n\otimes\cD_{1n}}\oplus u_2^*\big)(I_n \otimes P_1^*) = E_{1n} \otimes \cD_{\tilde{T_n}}$ which lies in the kernel of $P_2^*$ and hence $P_2^*\big(I_{E_n\otimes\cD_{1n}}\oplus u_2^*\big)(I_n\otimes P_1^*) = 0$.
	
	For the second term, we already have observed that 
	\[
	\big(I_{1n}\otimes P_2^*\big(I_{E_n\otimes\cD_{1n}}\oplus u_2^*\big)(I_n\otimes i_2)\big)(I_n\otimes j_2^{'*}) = I_{1n}\otimes \Big(P_2^*\big(I_{E_n\otimes\cD_{1n}}\oplus u_2^*\big)(I_n\otimes i_2) (I_{E_n\otimes \Dn }  \oplus u_2 ) \bar{i}_2^{'*} \Big).
	\]
	Now,
	\begin{align*}
		\big(I_{E_n\otimes\cD_{1n}}\oplus u_2^*\big)(I_n\otimes i_2) (I_{E_n\otimes \Dn }  \oplus u_2 )=&\big(I_{E_{1n}\otimes \Done }\oplus I_{E_n\otimes \Dn }\oplus u_2^*\big)(I_n\otimes i_2) (I_{E_n\otimes \Dn }  \oplus u_2 )\\
		=&\big(0_{E_{1n}\otimes \Done }\oplus I_{E_n\otimes \Dn }\oplus u_2^*\big)(I_n\otimes i_2) (I_{E_n\otimes \Dn }  \oplus u_2 )\\
		=&i_{(E_n\otimes \Dn )\oplus \cE_2}
	\end{align*}
	where $i_{(E_n\otimes \Dn )\oplus \cE_2}:(E_n\otimes \Dn )\oplus \cE_2\to \big( E_n \otimes \cD_{\tilde{T}_1} \big) \oplus \big( E_{1n} \otimes \cD_{\tilde{T_n}} \big)  \oplus \cE_2$ is an inclusion and the last line follows since $I_n \otimes i_2$ is an inclusion map with range $E_n \otimes \big( \cD_{\tilde{T}_1} \oplus \cE_1 \big)$. 
	So,
	\[
	P_2^*\big(I_{E_n\otimes\cD_{1n}}\oplus u_2^*\big)(I_n\otimes i_2) (I_{E_n\otimes \Dn }  \oplus u_2 ) \bar{i}_2^{'*}=P_2^* (i_{(E_n\otimes \Dn )\oplus \cE_2}) \bar{i}_2^{'*}.
	\]
	Observe that,
	\[
	(E_n \otimes \cD_{\tilde{T}_1})\oplus\cD_{\tilde{T_n}} \oplus \cE_2\xrightarrow{\bar{i}_2^{'*}} (E_n \otimes \cD_{\tilde{T}_1}) \oplus \cE_2\xrightarrow{i_{(E_n\otimes \Dn )\oplus \cE_2}} \big( E_n \otimes \cD_{\tilde{T}_1} \big) \oplus \big( E_{1n} \otimes \cD_{\tilde{T_n}} \big)  \oplus \cE_2\xrightarrow{P_2^*} (E_n \otimes \cD_{\tilde{T}_1})\oplus \cE_2.
	\]
	So,
	\[
	P_2^* (i_{(E_n\otimes \Dn )\oplus \cE_2})\bar{i}_2^{'*} =P_{(E_n \otimes \cD_{\tilde{T}_1})\oplus \cE_2}.
	\]
	and hence 
	\[ 
	I_{1n}\otimes \Big(P_2^*(I_n\otimes i_2) (I_{E_n\otimes \Dn }  \oplus u_2 ) \bar{i}_2^{'*} \Big) = P_{E_{1n} \otimes \big((E_n \otimes \cD_{\tilde{T}_1})\oplus \cE_2\big)}.
	\]
	Now, for the third term, we consider the following diagram 
	\[
	\begin{CD}
		E_{1n}\otimes \Done \xrightarrow{I_n \otimes P_1^*} E_n \otimes \big(\Dn  \oplus ( E_{1} \otimes \Done  )  \oplus  \cE_1 \big)\xrightarrow{i_1^{'*}} \big( E_{1n} \otimes \Done  \big) 
		\xrightarrow{I_{1n}\otimes i_1}\ E_{1n} \otimes \big(\Done  \oplus( E_n\otimes \Dn )\oplus \cE_2 \big) 
	\end{CD} 
	\] \\
	with the help of range$(I_n \otimes P_1^*) = E_{1n} \otimes \Done $ it is easy to observe that 
	\[ (I_{1n}\otimes i_1)i_1^{'*}(I_n\otimes P_1^*) = P_{E_{1n} \otimes \Done }. \]
	For the last term, first note that 
	$\text{ran}\, j_2^{'*} = E_{1n} \otimes (\Dn  \oplus \cE_1)$, so $\text{ran}(I_{1n}\otimes i_2) j_2^{'*} = E_{1n} \otimes (\Dn  \oplus \cE_1)$ and hence $\text{ran}\big(I_n\otimes (I_{1n}\otimes i_2) j_2^{'*}\big) = E_{1n}\otimes E_n \otimes (\Dn  \oplus \cE_1)$.  On the other hand, $E_n \otimes (\Dn  \oplus \cE_1)\subseteq \text{ker}\, i_1^{'*}$ which implies  $E_{1n}\otimes E_n \otimes (\Dn  \oplus \cE_1)\subseteq \text{ker}(I_{1n}\otimes i_1^{'*})$. Therefore,
	\begin{align*}
		\big( I_{1n} \otimes I_{1n} \otimes U \big) &\Big( I_{1n} \otimes (I_{1n} \otimes i_1) i_1^{'*}  \Big) \Big( I_n \otimes \big( I_{1n} \otimes  i_2  \big)  j_2^{'*} \Big) \big(I_{1n} \otimes U^* \big)\\=&\big( I_{1n} \otimes I_{1n} \otimes U \big)(I_{1n} \otimes I_{1n} \otimes i_1)(I_{1n}\otimes i_1^{'*}) \Big( I_n \otimes \big( I_{1n} \otimes  i_2  \big)  j_2^{'*} \Big) \big(I_{1n} \otimes U^* \big)=0
	\end{align*}
	
	Therefore, 
	\begin{align*}
		\tau_{U_n^*}(I_n\otimes\tau_{U_1^*})=&I_{\cF(E)} \otimes (I_{1n}\otimes U) 
		\Big[0  + P_{E_{1n} \otimes \big((E_n \otimes \cD_{\tilde{T}_1})\oplus \cE_2\big)}  + 
		P_{E_{1n}\otimes \cD_{\tilde{T_n}}} + 0 \Big](I_{1n}\otimes U^*) \\
		=&I_{\cF(E)}\otimes(I_{1n}\otimes U) 
		\Big[ I_{E_{1n} \otimes \big((E_n\otimes \cD_{\tilde{T}_1})\oplus\cD_{\tilde{T_n}}\oplus \cE_2\big)} \Big](I_{1n}\otimes U^*)  \\
		=&I_{\cF(E)}\otimes I_{E_{1n} \otimes \big((E_1\otimes \cD_{\tilde{T_n}})\oplus\cD_{\tilde{T}_1}\oplus \cE_1\big)}\\
		=&\tilde{L}_1
	\end{align*}
	
\end{proof}
\begin{theorem}\label{maindilation} Let $(\sigma,T_1,\ldots,T_n)\in\cT_{1,n}^n(\sigma,\,M,\,E_1,\ldots,E_n,H)$ and let $E$, $\cD$, $\Pi_{1n,V}:H\to \cF(E)\otimes \cD$ and $L_i\, (i=2,\ldots,n)$ as above. Then $\Pi_{1n,V}$ is an isometry and 
	\[\begin{cases}
		(I_1\otimes\Pi_{1n,V})\tilde{T_1}^*=\tau_{U_1^*}^*\Pi_{1n,V}\\
		(I_i\otimes\Pi_{1n,V})\tilde{T_i}^*=\tilde{L}_i^*\Pi_{1n,V}\quad \text{for}\,\,\,2\leq i\leq n-1\\
		(I_n\otimes\Pi_{1n,V})\tilde{T_n}^*=\tau_{U_n^*}^*\Pi_{1n,V}
	\end{cases}
	\]
	where $\tau_{U_1^*}$ and $\tau_{U_n^*}$ are isometric multipliers that have transfer function representations with respect to some unitaries $U_1$ and $U_n$. Moreover, $\tau_{U_1^*}$ and $\tau_{U_n^*}$ satisfy the following relation:
	\[
	\tau_{U_1^*}(I_1\otimes\tau_{U_n^*})=\tilde{L}=\tau_{U_n^*}(I_n\otimes\tau_{U_1^*}).
	\]
	and there is an isometric representation $(\rho, M_1, L_2,\ldots,L_{n-1},M_n)$ of the product system that dilates the representation $(\sigma,T_1,\ldots,T_n)$ and such that $\tau_{U_i^*}=\tilde{M_i}$ for $i=1,n$.
\end{theorem}

\end{section}

\begin{section}{Applications}
	\begin{theorem}\label{appl1}
		Let $M$ be a von Neumann algebra and $\alpha_1,\alpha_2,\ldots,\alpha_n$ be $*-$automorphisms of $M$ that pairwise commute. Also, let $\sigma$ be a normal representation of $M$ on a Hilbert space $H$ and let $t:=(t_1,t_2,\ldots,t_n )$ be a commuting tuple of operators in B(H) such that
		\begin{enumerate} 
			\item[(1)] $t$ satisfies the conditions from \cite{BDHS19} : $\hat{t}_1$ and $\hat{t}_n$ satisfy Szego positivity and $\hat{t}_n$ is pure.
			\item[(2)] For every $a\in M$ and $1\leq i \leq n$,
			\begin{equation} \label{alpharep} 
				t_i\sigma(\alpha_i(a))=\sigma(a)t_i
			\end{equation} 
		\end{enumerate}  
		Then there is a Hilbert space $K$ , an isometry $V:H\rightarrow K$ , a representation $\rho$ of $M$ on $K$ and commuting isometries  $v_1,\ldots.v_n$ in $B(K)$ that satisfy
		\begin{equation} 
			v_i\rho(\alpha_i(a))=\rho(a)v_i
		\end{equation} and
		dilate $(\sigma,t_1,\ldots,t_n)$, i.e,
		\begin{enumerate} 
			\item[(i)] $\sigma(a)=V^*\rho(a)V$ for all $a\in M$ and
			\item[(ii)] $Vt_i^*=v_i^*V$ for all $i$
		\end{enumerate} 
		
	\end{theorem}	
	Setting $\alpha_i=id_M$ above, we get
	
	\begin{corollary}
		Let $M$ be a von Neumann algebra and let $\sigma$ be a normal representation of $M$ on a Hilbert space $H$ and  $t:=(t_1,t_2,\ldots,t_n )$ be a commuting tuple of operators in B(H) such that
		\begin{enumerate} 
			\item[(1)] $t$ satisfies the conditions from \cite{BDHS19} : $\hat{t}_1$ and $\hat{t}_n$ satisfy Szego positivity and $\hat{t}_n$ is pure.
			\item[(2)] For every  $1\leq i \leq n$, $t_i\in \sigma(M)'$.
		\end{enumerate}  
		Then there is a Hilbert space $K$ , an isometry $V:H\rightarrow K$ , a representation $\rho$ of $M$ on $K$ and commuting isometries  $v_1,\ldots.v_n$ in $\rho(M)'$ that  
		dilate $(\sigma,t_1,\ldots,t_n)$, i.e,
		\begin{enumerate} 
			\item[(i)] $\sigma(a)=V^*\rho(a)V$ for all $a\in M$ and
			\item[(ii)] $Vt_i^*=v_i^*V$ for all $i$
		\end{enumerate} 
		
	\end{corollary}
	
	To prove the theorem, we  explain how Theorem~\ref{appl1} is related to Theorem~\ref{maindilation}.
	
	For this, we first note that, given an automorphism $\alpha$ of a von Neumann algebra $M$, one can associate a $W^*$-correspondence, written $_{\alpha}M$, with it. As a vector space $_{\alpha}M$ is $M$, The $M$-valued inner product is $\langle a,b \rangle =a^*b$, the left multiplication is $m \cdot a=ma$ and the right multiplication is $\varphi(m)a=\alpha(m)a$ (where $m,a $ are in $M$ but $a$ here is viewed as an element of the correspondence). 
	
	If $\alpha, \beta$ are two automorphisms of $M$ then
	$$_{\alpha}M\otimes _{\beta}M \cong _{\beta \alpha}M $$
	where the isomorphism is given by $a\otimes b \mapsto \beta(a)b$. Thus we will identify $_{\alpha}M\otimes _{\beta}M $ with $ _{\beta \alpha}M $ (and, when $\alpha$ and $\beta$ commute, we identify $_{\alpha}M\otimes _{\beta}M $ with $_{\beta}M\otimes _{\alpha}M $.
	
	Now let $(T,\sigma)$ be a representation of $_{\alpha}M$. We have, for $a,b\in M$,
	$T(\alpha(a)b)=T(\varphi(a)b)=\sigma(a)T(b)$. Setting $b=I$, we get $T(\alpha(a))=\sigma(a)T(I)$. But $T(\alpha(a))=T(I\cdot \alpha(a))=T(I)\sigma(\alpha(a))$ and, writing $t:=T(I)$ , we get
	$t\sigma(\alpha(a))=\sigma(a)t$ for every $a\in M$. So, a representation of this correspondence is given by a normal representation $\sigma$ on $H$ and a contraction $t\in B(H)$ satisfying the intertwining relation
	\begin{equation}
		t\sigma(\alpha(a))=\sigma(a)t
	\end{equation} for all $a\in M$. (See \cite{MS06} for details).
	Note also that $_{\alpha}M\otimes_{\sigma}H\cong H$ via the isomorphism $a\otimes h\mapsto \sigma(a)h$. Identifying $_{\alpha}M\otimes_{\sigma}H$ with $H$, we find that $\tilde{T}$ can be viewed as an operator in $B(H)$. With these identifications, we write $$\tilde{T}=t.$$ In fact, $t=\tilde{T}(I\otimes h)$ and $t\sigma(a)h=\tilde{T}(I\otimes_{\sigma}\sigma(a)h)=\tilde{T}(a\otimes h)$.
	
	Now suppose that $(\sigma,T)$ is a representation of $_{\alpha}M$ and $(\sigma,S)$ is a representation of $_{\beta}M$. Recall that, when we identify $_{\alpha}M\otimes _{\beta}M$ with $_{\beta \alpha}M$, $I_{_{\beta \alpha}M}$ is identified with $I_{_{\alpha}M}\otimes I_{_{\beta}M}$ and, thus,
	\begin{equation}\label{ST}    \tilde{S}(I_{_{\beta}M}\otimes \tilde{T})(I_{_{\beta \alpha}M}\otimes h)=\tilde{S}(I_{_{\beta}M}\otimes \tilde{T})(I_{_{\alpha}M}\otimes I_{_{\beta}M}\otimes h)=\tilde{S}(I_{_\beta M}\otimes th)=sth.\end{equation}
	Thus, $\tilde{S}(I_{_{\beta}M}\otimes \tilde{T})$ is associated with $st$.
	A straightforward computation using the identifications above shows that
	$\tilde{T}^*k=t^*k$ and
	\begin{equation}\label{st*}\tilde{S}\tilde{T}^*=st^*.\end{equation}
	
	For every representation $(\sigma,T)$,
	$$\tilde{T}^*\tilde{T}(a\otimes h)=\tilde{T}^*(t\sigma(a)h)=t^*t\sigma(a)h $$ and we conclude that $(\sigma,T)$ is isometric if and only if $t$ is an isometry.
	
	We now turn to prove Theorem~\ref{appl1}. 
	
	\begin{proof} 
		
		We now fix the setting of Theorem~\ref{appl1}. That is, we fix a von Neumann algebra  $M$  and pairwise commuting automorphisms $\alpha_1,\alpha_2,\ldots,\alpha_n$ on $M$. Also, we fix a normal representation  $\sigma$  of $M$ on a Hilbert space $H$ and let $t:=(t_1,t_2,\ldots,t_n )$ be a commuting tuple of operators in $B(H)$ such that
		\begin{enumerate} 
			\item[(1)] $t$ satisfies the conditions from \cite{BDHS19} : $\hat{t}_1$ and $\hat{t}_n$ satisfy Szego positivity and $\hat{t}_n$ is pure.
			\item[(2)] For every $a\in M$ and $1\leq i \leq n$,
			\begin{equation} \label{alpharep} 
				t_i\sigma(\alpha_i(a))=\sigma(a)t_i
			\end{equation} 
		\end{enumerate}   
		
		Since $\{\alpha_i\}$ pairwise commute, we can identify 	 $_{\alpha_i}M\otimes _{\alpha_j}M $ with $_{\alpha_j}M\otimes _{\alpha_i}M $ (as above) and $(_{\alpha_1}M,\ldots,_{\alpha_n}M)$ defines a product system. Equation (\ref{alpharep}) implies that each $t_i$ gives rise to a representation $(\sigma, T_i)$ of $_{\alpha_i}M$ and the associated $\tilde{T_i}$ satisfies $t_i=\tilde{T_i}(I\otimes h)$ and $t_i\sigma(a)h=\tilde{T_i}(I\otimes_{\sigma}\sigma(a)h)=\tilde{T_i}(a\otimes h)$.  Equation (\ref{ST}) shows that these representations pairwise commute and, thus, define a representation of the product system.
		
		Since the Szego condition and the pureness conditions involve expressions of the form $t_it_j^*$ (for operators) or $\tilde{T}_i\tilde{T_j}^*$ (for representations), it follows from (\ref{st*}) that (1) in Theorem~\ref{appl1} implies that $(\sigma,T_1,\ldots,T_n)\in\cT_{1,n}^n(\sigma,\,M,\,E_1,\ldots,E_n,H)$ and, thus, Theorem~\ref{maindilation} applies.
		
		We conclude that there is an isometric representation $(\rho, M_1, L_2,\ldots,L_{n-1},M_n)$ of the product system that dilates the representation $(\sigma,T_1,\ldots,T_n)$ and such that $\tau_{U_i^*}=\tilde{M_i}$ for $i=1,2$.
		
		Writing the isometric representation $(\rho, M_1, L_2,\ldots,L_{n-1},M_n)$ as $(\rho, V_1, V_2,\ldots,V_{n-1},V_n)$ (and let $v_i=V_i(I)$ as above), this completes the proof of the theorem.
	\end{proof}
	
	\begin{remark}
		\begin{enumerate}
			\item [1.] The isometries $v_i$ can be written explicitely as shown in Theorem~\ref{maindilation}. 
			\item[2.] For $M=\bC$ (and writing $v_i$ explicitely as noted in 1.) we get the main result of \cite{BDHS19}.
		\end{enumerate}
	\end{remark}
	
\end{section}

 \noindent\textbf{Acknowledgements:} The first and second named authors are grateful to Prof. Bata Krishna Das for some insightful discussions at the beginning of this project. Part of the first named author's research was carried out at the Ben-Gurion University and the research is partially supported by the Israel Science Foundation grant number 2196-20. The second named author's research is supported by the DST-INSPIRE Faculty Fellowship No. DST/INSPIRE/04/2020/001250. 
 
\end{document}